\documentclass[reqno]{amsart}
\usepackage{amsmath}%
\usepackage{amsfonts}%
\usepackage{amssymb}%
\usepackage{graphicx}
\newtheorem{theorem}{Theorem}
\theoremstyle{plain}

\newtheorem{corollary}{Corollary}

\newtheorem{definition}{Definition}
\newtheorem{example}{Example}

\newtheorem{lemma}{Lemma}

\newtheorem{proposition}{Proposition}
\newtheorem{remark}{Remark}

\numberwithin{equation}{section}
\begin{document}
\begin{center}
\vspace*{1.3cm}

\textbf{FUNCTIONS WITH UNIFORM LEVEL SETS}

\bigskip

by

\bigskip

PETRA WEIDNER\footnote{HAWK Hildesheim/Holzminden/G\"ottingen University of Applied Sciences and Arts, Faculty of Natural Sciences and Technology,
D-37085 G\"ottingen, Germany, {petra.weidner@hawk-hhg.de}.}

\bigskip
\bigskip
Research Report \\ 
Version 3 from August 09, 2016\\
Extended Version of Version 1 from January 30, 2016

\end{center}

\bigskip
\bigskip

\noindent{\small {\textbf{Abstract:}}
Functions with uniform level sets can represent orders, preference relations or other binary relations and thus turn out to be a tool for scalarization that can be used, e.g., in multicriteria optimization, decision theory, mathematical finance, production theory and operator theory. Sets which are not necessarily convex can be separated by functions with uniform level sets. This has a deep impact on functional analysis, where many proofs require separation theorems. 
This report focuses on properties of real-valued and extended-real-valued functions with uniform level sets which are defined on a  topological vector space. This includes the extension of aspects and results given in an earlier paper by Gerth (now Tammer) and Weidner. The functions may be, e.g., continuous, convex, strictly quasiconcave or sublinear. They can coincide with a Minkowski functional or with an order unit norm on a subset of the space.
As a side result, we show that the core of a closed pointed convex cone is its interior in an appropriate norm topology.
}

\bigskip


\bigskip


\section{Introduction}

In this paper, we investigate functions $\varphi _{A,k}$ with uniform level sets based on the formula
\begin{equation}
\varphi_{A,k} (y):= \inf \{t\in
{\mathbb{R}} \mid y\in tk + A\}, \label{first-funcak0}
\end{equation}
where $A$ is a subset of a topological vector space $Y$ and $k\in Y\setminus\{0\}$.

This formula was introduced by Tammer (formerly Gerstewitz and Gerth) for convex sets $A$ under more restrictive assumptions in the context of vector optimization \cite{ger85}.
Basic properties of $\varphi _{A,k}$ have been proved in \cite{GerWei90} and \cite{Wei90}, later followed by \cite{GopRiaTamZal:03},
\cite{TamZal10} and \cite{DT}. For detailed bibliographical notes, see Section \ref{sec-phiAk_biogr}. There we will also point out that researchers of different fields of mathematics and economic theory have applied functions of type $\varphi _{A,k}$ since these functions are appropriate for separating nonconvex sets and for the scalarization of relations like partial orders or preference relations.

Depending on the choice of $A$ and $k$, $\varphi _{A,k}$ can be real-valued or also attain the value $-\infty$.
We will use the symbolic function value $\nu $ (instead of the value $+\infty$ in convex analysis) when extending a functional to the entire space or at points where a function is not feasible otherwise. Thus our approach differs from the classical one in convex analysis in these cases since the functions we are studying are of interest in minimization problems as well as in maximization problems. Consequently, we consider functions which can attain values in $\overline{\mathbb{R}}_{\nu }:=\overline{\mathbb{R}}\cup\{\nu \}$, where $\overline{\mathbb{R}}:=\mathbb{R}\cup\{-\infty, +\infty\}$. $\varphi _{A,k}$ never attains the value $+\infty$ since we define $\operatorname*{sup}\emptyset = \operatorname*{inf}\emptyset =\nu $. Details of functions with values in $\overline{\mathbb{R}}_{\nu }$ are explained in \cite{Wei15}. 
For the application of this approach to $\varphi _{A,k}$ we have to keep in mind the following terms and definitions:
\begin{enumerate}
\item $\operatorname*{inf}\emptyset =\nu\not\in\overline{\mathbb{R}}$
\item $\operatorname*{dom}\varphi _{A,k}=\{y\in Y\mid \varphi _{A,k}(y)\in\mathbb{R}\cup (-\infty )\}$ is the (effective) domain of $\varphi _{A,k}$
\item $\varphi _{A,k}$ is proper if $\operatorname*{dom}\varphi _{A,k}\not= \emptyset$ and $\varphi _{A,k}(y)\in\mathbb{R}
\;\;\forall y\in \operatorname*{dom}\varphi _{A,k}$
\item $\varphi _{A,k}$ is finite-valued if $\varphi _{A,k}(y)\in\mathbb{R} \;\;\forall y\in Y$ 
\end{enumerate}

We will start our investigations in Section \ref{sec-basicdef} with functions for which the sublevel sets are just linear shifts of a set $A$ into direction $k$ and $-k$, respectively. These functions turn out to be of type $\varphi_{A,k}$ with $k\in -0^+A\setminus\{0\}$, where $0^+A$ denotes the recession cone of $A$ defined below.
In the sections to follow, $\varphi_{A,k}$ is studied for closed sets $A$ in topological vector spaces with $k\in -0^+A\setminus\{0\}$. In this case, $\varphi_{A,k}$ is a lower semicontinuous translation-invariant function with uniform sublevel sets $A+tk$, $t\in\mathbb{R}$. $\varphi_{A,k}$ is finite-valued if $k\in -\operatorname*{core}0^+A$. These and further basic properties of functions of type $\varphi _{A,k}$ are proved in Section \ref{sec-basicprop}. Interdependencies between the functions $\varphi _{A,k}$, $\varphi _{A,\lambda k}$, $\varphi _{A+ck,k}$ and $\varphi _{y^0+A,k}$, which are essential for applications, are studied. 

We will always try to find conditions which are sufficient and necessary for certain properties of $\varphi _{A,k}$, e.g. for convexity, sublinearity, strict quasiconvexity, strict quasiconcavity, Lipschitz continuity.
Though we work in a topological vector space, $A$ does not have to contain interior points or algebraic interior points.
Assumptions are often formulated using the recession cone of $A$. We will show that these assumptions are equivalent to usual assumptions in production theory like the free-disposal assumption and the strong free-disposal assumption.

Theorem \ref{t-sep-all} and Theorem \ref{t-sep-H2} point out the way in which $\varphi _{A,k}$ separates sets.
Several statements connect $\varphi _{A,k}$ with the sublinear function $\varphi _{0^+A,k}$ (Proposition \ref{p-varphi_rec}, 
Proposition \ref{k-core}, Proposition \ref{prop-finite}, Theorem \ref{th-Lip-varphi}).

Section \ref{sec-Monot} deals with the monotonicity of $\varphi _{A,k}$ in the framework of scalarizing binary relations.

Just in the case that the domain of $\varphi _{A,k}$ is open and $\varphi _{A,k}$ is continuous on its domain, the level sets of $\varphi _{A,k}$ are the uniform sets $\operatorname*{bd}A+tk$, $t\in\mathbb{R}$. Functions with this property are studied more in depth in Section \ref{sec-basicprop-cont}. Proposition \ref{-phi_Ak} will give us a tool for transfering results for the minimization of functions of type $\varphi _{A,k}$ to the maximization of functions of this type. This is  of interest if $\varphi _{A,k}$ is a strictly quasiconcave utility function. In a Banach space, Lipschitz continuity of $\varphi _{A,k}$ is characterized by the condition $k\in -\operatorname*{int}0^+A$.

Section \ref{sec-cx} focuses on convex functions $\varphi _{A,k}$ including statements for sublinear functionals. In Proposition \ref{p-subdiff}, a necessary condition for subgradients of $\varphi _{A,k}$ is given by the sublinear function 
$\varphi _{0^+A,k}$.

In Proposition \ref{p-Mink-uni} and Proposition \ref{p-ordint-varphi}, we show the relationship between $\varphi _{A,k}$ and the Minkowski functional of $A+k$ if $A$ is a convex cone, and the coincidence of values of certain norms with values of $\varphi _{A,k}$. These norms are just order unit norms if the space is a Riesz space. Moreover, Section \ref{s-Mink-lev} contains a characterization of points from the core of a cone as interior points in a norm topology.

Finally, Section \ref{sec-phiAk_biogr} contains bibliographical notes which refer to the statements of this report and underline the connection with scalarization in vector optimization.

From now on, $\mathbb{R}$ and $\mathbb{N}$ will denote the sets of real numbers and of non-negative integers, respectively.
We define $\mathbb{R}_{+}:=\{x\in\mathbb{R}\mid x\geq 0\}$, $\mathbb{R}_{>}:=\{x\in\mathbb{R}\mid x > 0\}$,
$\mathbb{R}_{+}^2:=\{(x_1,x_2)\in\mathbb{R}^2\mid x_1\geq 0, x_2\geq 0\}$ and $\mathbb{N}_>:=\mathbb{N}\setminus\{0\}$.
Linear spaces will always be assumed to be real vector spaces. 
A set $C$ in a linear space $Y$ is a cone if $\lambda c\in C \;\forall\lambda\in\mathbb{R}_{+}, c\in C$. The cone $C$ is called non-trivial if $C\not=\{0\}$ and $C\not= Y$ hold. For a subset $A$ of some linear space $Y$, 
$\operatorname*{core}A$ will denote the algebraic interior of $A$ and $0^+A:=\{u\in Y  \mid  a+tu\in A\ \forall a\in A \; \forall t\in \mathbb{R}_{+}\}$ the recession cone of $A$.
Given two sets $A$, $B$ and some vector $k$ in $Y$, we will use the notation $A\; B:=A \cdot B:=\{a \cdot b\mid a \in A , \; b\in B\}$ and $A\; k:=A \cdot k:=A \cdot \{ k\}$.
In a topological space $Y$, $\operatorname*{cl}A$, $\operatorname*{int}A$ and $\operatorname*{bd}A$ denote the closure, the interior and the boundary, respectively, of a subset $A$.
For a functional $\varphi$ defined on some space $Y$ and attaining values in $\overline{\mathbb{R}}_{\nu }$ we will denote the epigraph of $\varphi$ by   
$\operatorname*{epi}\varphi$, the effective domain of $\varphi$ by $\operatorname*{dom}\varphi$ and the level sets of $\varphi$ w.r.t. some binary relation $\mathcal{R}$ given on $\overline{\mathbb{R}}_{\nu }$ by $\operatorname*{lev}_{\varphi,\mathcal{R}}(t):= \{y\in Y \mid \varphi (y) \mathcal{R} t\}$ with $t\in\mathbb{R}$.
Note that -- for simplicity -- we use the notion level set not only for the relation $=$, but e.g. also for relations $\leq$ and $>$. 

Beside the properties of functions defined in \cite{Wei15}, we will need the following ones:
\begin{definition}\label{d-mon}
Let $Y$ be a linear space, $B\subseteq Y$ and $\varphi: Y \to \overline{\mathbb{R}}_{\nu } $. \\
$\varphi$ is said to be
\begin{itemize}
\item[(a)]
$B$-monotone on $F\subseteq \operatorname*{dom}\varphi$
if $y^1,y^2 \in F$ and $y^{2}-y^{1}\in B$ imply $\varphi
(y^{1})\le \varphi (y^{2})$,
\item[(b)] strictly $B$-monotone on $F\subseteq \operatorname*{dom}\varphi$ 
if $y^1,y^2 \in F$ and $y^{2}-y^{1}\in B\setminus
\{0\}$ imply $\varphi (y^{1})<\varphi (y^{2})$,
\item[(c)] $B$-monotone or strictly $B$-monotone if it is $B$-monotone or strictly $B$-monotone, respectively, on $\operatorname*{dom}\varphi$,
\item[(d)] quasiconvex if $\operatorname*{dom}\varphi$ is convex and
\[\varphi (\lambda y^1 + (1-\lambda ) y^2) \le \operatorname*{max}(\varphi (y^1),\varphi (y^2))\]
for all $y^1, y^2 \in \operatorname*{dom}\varphi$ and $\lambda \in (0,1)$,
\item[(e)] strictly quasiconvex if $\operatorname*{dom}\varphi$ is convex and
\[\varphi (\lambda y^1 + (1-\lambda ) y^2) < \operatorname*{max}(\varphi (y^1),\varphi (y^2))\]
for all $y^1, y^2 \in \operatorname*{dom}\varphi$ with $y^1\not= y^2$ and $\lambda \in (0,1)$,
\item[(f)] strictly quasiconcave if $-\varphi$ is strictly quasiconvex. 
\end{itemize}
\end{definition}
\smallskip

\section{Definition of Functions with Uniform Level Sets}\label{sec-basicdef}

A functional $\varphi$ separates two sets $V$ and $W$ in a space $\,Y$ if there exists some value $t\in \mathbb{R}$ such that one of the sets is contained in $M:=\{y\in Y \mid \varphi (y) \le t \}$, the other one is contained in $\{y\in Y \mid \varphi (y) \ge t \}$ and $V\cup W\not\subseteq \{y\in Y \mid \varphi (y) = t \}$. 
Disjoint convex sets in a finite-dimensional vector space can be separated by some linear functional $\varphi$. In this case, $M=tk+A$ for some halfspace A and some $k\in Y$. Being interested in nonconvex sets, we use this idea and investigate functionals $\varphi $ which fulfill the condition
\begin{displaymath}
\varphi(y)\le t \iff y\in tk+A.
\end{displaymath}
Here, A is assumed to be some proper subset of Y.

\begin{proposition} \label{prop-vor-theo}
Consider a linear space Y and a function $\varphi:Y\rightarrow \overline{{\mathbb{R}}}_{\nu }$ for which there exist some proper subset $A$ of $\; Y$ and some $k\in Y\setminus\{0\}$
such that $\operatorname*{dom}\varphi =\mathbb{R}k+A$ and
\begin{equation}
\operatorname*{lev}\nolimits_{\varphi,\le}(t)=t k+A\quad \forall\, \ t \in {\mathbb{R}}. \label{K4}
\end{equation}
Then 
\begin{equation}\label{d-sep-func}
k\in -0^+A 
\end{equation}
and
\begin{equation}
\varphi (y)= \inf \{t\in
{\mathbb{R}} \mid y\in tk + A\} \;\;\forall y\in Y \label{vor-funcak0}
\end{equation}
hold. 
\end{proposition}
\begin{proof}
$y\in A-tk$, $t\in\mathbb{R}_{+}$, and (\ref{K4}) result in $\varphi(y)\le -t \le 0$ and thus in $y \in A$. Hence (\ref{d-sep-func}) is satisfied.
(\ref{K4}) and $\operatorname*{dom}\varphi =\mathbb{R}k+A$ imply $\varphi (y)\not= +\infty\;\forall y\in Y$.
If $\varphi (y)=-\infty$, then $y\in tk+A\; \forall\, \ t \in {\mathbb{R}}$, thus $\inf \{t\in
{\mathbb{R}} \mid y\in tk + A\}=-\infty$. If $\varphi (y)=t\in {\mathbb{R}}$, then $y\in tk+A$. If (\ref{vor-funcak0}) would not be satisfied, then there would exist some 
$\lambda \in {\mathbb{R}}$ with $\lambda < t$ and 
$y\in \lambda k+A$. This would imply $\varphi (y)\le\lambda < t$, a contradiction.
\end{proof}

Let us note that $k\in -0^+A$ is equivalent to $A=A-\mathbb{R}_{+} k$ and results in $\operatorname{int} A=\operatorname{int} A-\mathbb{R}_{+} k$.

We will show that (\ref{d-sep-func}) and (\ref{vor-funcak0}) imply (\ref{K4}) if $Y$ is a topological vector space and $A$ is closed. This implication is not true for an arbitrary set $A$.
\begin{example}
Take $Y=\mathbb{R}^{2}$, $k:=(1,0)^T$ and
$A:=\{ (y_1,y_2)^T\in\mathbb{R}^{2}\mid y_1< 0\}$. (\ref{d-sep-func}) is satisfied. Define $\varphi:Y\rightarrow \overline{{\mathbb{R}}}_{\nu }$ by (\ref{vor-funcak0}).
Then $\varphi((0,0)^T)=0$ though $(0,0)^T\not\in A$. Thus (\ref{K4}) is not fulfilled.\\
Moreover, $A$ and k can not be used to define a functional $\varphi$ by condition (\ref{K4}) since, in this case, $\varphi((0,0)^T)=0$ because of $\varphi((0,0)^T)\le t
\;\forall t>0$, but $(0,0)^T\not\in A$.
\end{example}

\begin{definition}\label{d-funcak0}
Given a linear space Y, some proper subset $A$ of $\; Y$ and $k\in Y\setminus\{0\}$,
the function $\varphi _{A,k}:Y\rightarrow \overline{{\mathbb{R}}}_{\nu }$ is defined by
\begin{equation}
\varphi_{A,k} (y):= \inf \{t\in
{\mathbb{R}} \mid y\in tk + A\}. \label{funcak0}
\end{equation}
\end{definition}

One gets an immediate geometric interpretation of $\varphi_{A,k}$ since $tk + A$ is just the set $A$ shifted by $tk$.

Before we restrict our attention to closed sets $A$ in topological vector spaces, let us note that $\varphi_{A,k}$ can differ from $\varphi_{\operatorname*{cl}A,k}$ even on $\operatorname*{dom}\varphi _{A,k}$. In the first example, we will illustrate this for a finite-valued functional, in the second one for a non-trivial pointed convex cone $A$.
\begin{example}
Consider $Y=\mathbb{R}^{2}$, $k:=(1,0)^T$ and
$A:=\{ (y_1,y_2)^T\in\mathbb{R}^{2}\mid y_1\le 0\} \cup \{ (y_1,y_2)^T\in\mathbb{R}^{2}\mid y_1\le 1, y_2< 0\}$. $\varphi _{A,k}$ and $\varphi _{\operatorname*{cl}A,k}$ are finite-valued. $\varphi _{A,k}((0,0)^T)=0$, but $\varphi _{\operatorname*{cl}A,k}((0,0)^T)=-1$. 
\end{example}

\begin{example}
Consider $Y=\mathbb{R}^{2}$, $k:=(1,0)^T$ and
$A:=\{ (y_1,y_2)^T\in\mathbb{R}^{2}\mid y_2< 0\} \cup \{ (y_1,y_2)^T\in\mathbb{R}^{2}\mid y_1\le 0, y_2= 0\}$. $A$ is a non-trivial pointed convex cone.
$\varphi _{A,k}((0,0)^T)=0$, but $\varphi _{\operatorname*{cl}A,k}((0,0)^T)=-\infty$. 
\end{example}
\smallskip

\section{Basic Properties of Functions with Uniform Level Sets}\label{sec-basicprop}

We will now investigate basic properties of the functional.

\begin{theorem}
\label{t251}
Assume
\begin{description}
\item[$(H1_{A,k})$] $Y$ is a topological vector space, $A$ is a closed proper subset of $\;Y$ and
\begin{equation}
k\in -0^+A\setminus \{0\}. \label{r250}
\end{equation}
\end{description}

Then $\varphi _{A,k}$ is lower semicontinuous on $\operatorname*{dom}\varphi _{A,k}$,
\begin{eqnarray}
\operatorname*{dom}\varphi _{A,k} & = & {\mathbb{R}}k+A = {\mathbb{R}_{>}}k+A \not= \emptyset,\label{dom_Ak}\\
\operatorname*{lev}\nolimits_{\varphi_{A,k},\le}(t)& = & t k+A \quad \forall\, \ t \in {\mathbb{R}},\label{f-r252n}\\
tk+ \operatorname*{core}A & \subseteq & \operatorname*{lev}\nolimits_{\varphi_{A,k},<}(t)  \quad \forall\, t \in \mathbb{R},\label{core_in_less} \\
tk+ \operatorname*{int}A & \subseteq & \operatorname*{lev}\nolimits_{\varphi_{A,k},<}(t)  \quad \forall\, t \in \mathbb{R},\label{int_in_less} \\
\operatorname*{lev}\nolimits_{\varphi_{A,k},=}(t) & \subseteq & tk+ \operatorname*{bd}A  \quad \forall\, t \in\mathbb{R},\label{r0} 
\end{eqnarray} 
\begin{equation}
\operatorname*{epi}\varphi _{A,k} =\{(y,t)\in Y\times \mathbb{R} \mid y\in tk+A\}
\label{epi_Ak}
\end{equation}  
and
\begin{equation}
\varphi _{A,k} (y+t k)=\varphi _{A,k} (y)+t \quad \forall\, y\in Y,\, t \in
\mathbb{R}. 
\label{f-r255nn}
\end{equation}
Moreover:
\begin{itemize}
\item[(a)] The following conditions are equivalent to each other:
\begin{equation} \varphi _{A,k} \mbox{ is continuous on } \operatorname*{dom}\varphi _{A,k} \mbox{ and } \operatorname*{dom}\varphi _{A,k} \mbox{ is open},\end{equation} 
\begin{equation} A-\mathbb{R}_{>}\cdot k\subseteq \operatorname*{int} \;A, \label{vor-H2} \end{equation}   
\begin{equation} \operatorname*{bd} A-\mathbb{R}_{>}\cdot k\subseteq \operatorname*{int} \;A,\label{bd-H2} \end{equation} 
\begin{equation} \operatorname*{lev}\nolimits_{\varphi_{A,k},<}(t)=tk+ \operatorname*{int}A \quad \forall\, t \in \mathbb{R},\label{int_in_equal} \end{equation} 
\begin{equation} \operatorname*{lev}\nolimits_{\varphi_{A,k},=}(t) = tk+ \operatorname*{bd}A \quad \forall\, t \in \mathbb{R}.\label{r0_eq} \end{equation} 
Beside this, the following conditions are equivalent to each other:
\begin{equation} A-\mathbb{R}_{>}\cdot k\subseteq \operatorname*{core} \;A, \label{in_core} \end{equation} 
\begin{equation} \operatorname*{lev}\nolimits_{\varphi_{A,k},<}(t)=tk+ \operatorname*{core}A \quad \forall\, t \in \mathbb{R}.\label{core_in_equal} \end{equation} 
Condition (\ref{in_core}) implies $\operatorname*{dom}\varphi _{A,k}=\operatorname*{core}\operatorname*{dom}\varphi _{A,k}$.
\item[(b)] $\varphi _{A,k}(y)=-\infty \iff y+\mathbb{R}k\subseteq A$.
\item[(c)] $\varphi _{A,k}$ is finite-valued on $\operatorname*{dom}\varphi _{A,k}\setminus A$. 
\item[(d)] $\varphi _{A,k} $ is proper if and only if \\
$A$ does not contain
lines parallel to $k$, \textit{\textit{i.e.}}
\begin{equation}
\forall\, a\in A:\; a+\mathbb{R}k\not\subseteq A.\label{r255nx-inA}
\end{equation}
Condition (\ref{r255nx-inA}) is equivalent to
\begin{equation}
\forall\, y\in Y:\; y+\mathbb{R}k\not\subseteq A.\label{r255nx}
\end{equation}
\item[(e)] If $\varphi _{A,k} $ is proper, then 
\begin{equation}
A=\operatorname*{bd}A -\mathbb{R}_{+} k
\label{a_eq_bd}
\end{equation}
and
\begin{equation}
\operatorname*{dom}\varphi _{A,k}=\mathbb{R} k+\operatorname*{bd}A.
\label{dom_bd}
\end{equation}
\item[(f)] $\varphi _{A,k} $ is finite-valued if and only if 
$\varphi _{A,k} $ is proper and
\begin{equation}
Y=\mathbb{R} k+A. \label{r252nx}
\end{equation}
\item[(g)] $\varphi _{A,k} $ is quasiconvex $\iff$ $\varphi _{A,k} $ is convex $\iff$ $A$ is convex.
\item[(h)] $\varphi _{A,k} $ is positively homogeneous $\iff$ $A$ is a cone.\\
If $A$ is a cone and $k\in -\operatorname*{core} A$, then $\operatorname*{dom}\varphi _{A,k} =Y$ and $\operatorname*{lev}\nolimits_{\varphi_{A,k},<}(t)=tk+ \operatorname*{core}A \quad \forall\, t \in \mathbb{R}$.
\item[(i)] $\varphi _{A,k} $ is subadditive $\iff$ $A+A\subseteq A$.
\item[(j)] $\varphi _{A,k} $ is sublinear $\iff$ $A$ is a convex cone.
\end{itemize}
\end{theorem}

\begin{proof}
$\operatorname*{dom}\varphi _{A,k} ={\mathbb{R}}k+A= {\mathbb{R}_{>}}k+A$ and (\ref{f-r255nn}) follow immediately from the definition of $\varphi _{A,k}$.\\
Obviously, $\operatorname*{lev}_{\varphi_{A,k},\le}(t) \supseteq t k+A\;\; \forall\, \ t \in {\mathbb{R}}$.
We will now prove the inclusion $\subseteq$ of (\ref{f-r252n}).
If $\varphi_{A,k}(y)=\lambda <t$ for some $\lambda \in \mathbb{R} \cup \{-\infty\}$, then there exists some $\lambda_{1} \in (\lambda,t)$ with $y\in A+\lambda_{1} k=A+tk+(\lambda_{1} -t)k\subseteq A+tk$ because of (\ref{funcak0}) and (\ref{r250}) since $\lambda_{1}-t<0$.\\
For $\varphi_{A,k}(y)=t\in \mathbb{R}$, assume $y\notin tk+A$.
Then there exists some neighbourhood $U(y)\subset Y\setminus (tk+A)$ of y.
$\Rightarrow \exists \alpha\in \mathbb{R}:$ \quad $\alpha >0$ and $y-\lambda k\in U(y) \;\;\forall \lambda\in [0,\alpha )$. 
$\Rightarrow y-\lambda k\not\in tk+A \;\;\forall \lambda\in [0,\alpha )$.
$\Rightarrow y \not\in (t+\lambda ) k+A \;\;\forall \lambda\in [0,\alpha )$, a contradiction to (\ref{funcak0}).
Thus (\ref{f-r252n}) holds. This implies (\ref{epi_Ak}).\\
$\varphi_{A,k}$ is lower semicontinuous on $\operatorname*{dom}\varphi_{A,k}$ since the closedness of $A$ implies the closedness of all sets $\operatorname*{lev}_{\varphi_{A,k},\le}(t)$, $t \in {\mathbb{R}}$.\\
To prove (\ref{core_in_less}), consider some $y\in tk+ \operatorname*{core}A$, $t\in \mathbb{R}$.
$\Rightarrow \exists \lambda \in \mathbb{R}_{>}:\;\;\; y-tk+\lambda k\in A$. 
$\Rightarrow y\in (t-\lambda )k+A$.
$\Rightarrow \varphi _{A,k} (y)\le t-\lambda < t$ because of (\ref{f-r252n}). Thus (\ref{core_in_less}) holds.\\
(\ref{core_in_less}) implies (\ref{int_in_less}). \\
(\ref{r0}) follows immediately from (\ref{f-r252n}) and (\ref{int_in_less}).
\begin{itemize}
\item[(a)] Assume first that (\ref{vor-H2}) holds. Let $t \in
{\mathbb{R}}$ and $y\in Y$ be such that $\varphi_{A,k} (y)<t $.
Then there exists some $\lambda \in {\mathbb{R}}$, $\lambda < t $, such that $y\in
\lambda k+A$. It follows that $y\in \lambda k+A=t k+(A-(t
- \lambda )k)\subseteq t k+\mathop{\rm int}A$. This results, together with (\ref{int_in_less}), in 
(\ref{int_in_equal}).\\
Let us now assume that (\ref{int_in_equal}) is satisfied. Consider some $y\in A-\mathbb{R}_{>}\cdot k$. $\Rightarrow \varphi _{A,k}(y)< 0$. This implies $y\in \mathop{\rm int}A$ by (\ref{int_in_equal}). Thus (\ref{vor-H2}) is fulfilled. \\
(\ref{vor-H2}) is equivalent to (\ref{bd-H2}) because of (\ref{r250}). (\ref{int_in_equal}) is equivalent to (\ref{r0_eq}) because of (\ref{f-r252n}).\\
Consequently, (\ref{vor-H2}), (\ref{bd-H2}), (\ref{int_in_equal}) and (\ref{r0_eq}) are equivalent. \\
Assume that (\ref{int_in_equal}) holds. This implies the upper semicontinuity and continuity of $\varphi _{A,k}$ on $\operatorname*{dom}\varphi _{A,k}$.
Consider some arbitrary $y\in\operatorname*{dom}\varphi _{A,k}$. $\Rightarrow \exists a\in A, t\in \mathbb{R}: y=a+tk$.
Then for each $\lambda\in \mathbb{R}_{>}$: $y=a-\lambda k+(\lambda +t)k\in \mathop{\rm int}A+(\lambda +t)k=\mathop{\rm int}(A+(\lambda +t)k)\subseteq \mathop{\rm int}(A+\mathbb{R}k)$.
Thus $\operatorname*{dom}\varphi _{A,k}=A+\mathbb{R}k$ is open.\\
If $\operatorname*{dom}\varphi _{A,k}$ is open and $\varphi _{A,k} $ is continuous on $\operatorname*{dom}\varphi _{A,k}$, then $\varphi _{A,k} $ is upper semicontinuous on $\operatorname*{dom}\varphi _{A,k}$
and $\operatorname*{lev}_{\varphi_{A,k},<}(t)$ is open for each $t\in\mathbb{R}$. This implies (\ref{int_in_equal}) since
$tk+ \operatorname*{int}A\subseteq \operatorname*{lev}_{\varphi_{A,k},<}(t) \subseteq tk+A$ for each $t\in\mathbb{R}$.\\
The equivalence of (\ref{in_core}) and (\ref{core_in_equal}) follows in an analogous way as the equivalence of (\ref{vor-H2}) and (\ref{int_in_equal}).\\
If (\ref{in_core}) holds, then $\mathbb{R} k+A\subseteq \mathbb{R} k+\operatorname*{core}A\subseteq \operatorname*{core}(\mathbb{R} k+A)$, i.e.
$\operatorname*{dom}\varphi _{A,k} \subseteq\operatorname*{core}\operatorname*{dom}\varphi _{A,k}$.
\item[(b)] Consider $y\in Y$. 
$y+\mathbb{R}k\subseteq A \Leftrightarrow \forall t\in\mathbb{R}:\; y+tk\in A \Leftrightarrow \forall t\in\mathbb{R}:\; y\in -tk+A$
$\;\Leftrightarrow \forall t\in\mathbb{R}:\;\varphi_{A,k}(y)\le -t \Leftrightarrow \varphi_{A,k}(y)=-\infty$.
\item[(c)] (b) implies for $y\in Y$ with $\varphi_{A,k}(y)=-\infty$: $\;y=y+0\cdot k\in A$.
\item[(d)] Because of (b), $\varphi_{A,k}$ is proper $\Leftrightarrow$ (\ref{r255nx}).
(\ref{r255nx}) is equivalent to (\ref{r255nx-inA}) since $y+\mathbb{R}k\subseteq A$ implies $y=y+0\cdot k\in A$.
\item[(e)] Consider some arbitrary $a\in A$. Because of (\ref{r255nx-inA}), there exists some $t\in \mathbb{R}$ such that $a+tk\notin A$. 
$(H1_{A,k})$ implies $t>0$. $\Rightarrow \exists \lambda \in (0,1]:\quad \lambda a+(1-\lambda )(a+tk)\in \operatorname{bd} A$, 
i.e. $a+(1-\lambda )tk \in \operatorname{bd} A$. $\Rightarrow a\in \operatorname{bd} A-\mathbb{R}_{+} k$. Consequently, $A\subseteq \operatorname*{bd}A -\mathbb{R}_{+} k$.
This and $(H1_{A,k})$ yield (\ref{a_eq_bd}).
Because of (\ref{dom_Ak}) we get (\ref{dom_bd}).
\item[(f)] results from (d) and $\operatorname*{dom}\varphi_{A,k} ={\mathbb{R}}k + A$.
\item[(g)] Suppose first that $A$ is convex. Then $\operatorname*{dom}\varphi _{A,k}$ is convex. Take $(y^{1},t_1)$, $(y^{2},t_2)\in \operatorname*{epi}\varphi _{A,k}$ 
 and $\lambda\in[0,1]$. $\Rightarrow y^i\in t_{i}k+A$ for $i\in \{1,2\}$. 
Then $\lambda y^{1}+(1-\lambda )y^{2}\in (\lambda t_{1}+(1-\lambda) t_{2})k+(\lambda A+ (1-\lambda) A)\subseteq (\lambda t_{1}+(1-\lambda) t_{2})k+A$.
$\Rightarrow \operatorname*{epi}\varphi _{A,k}$ is convex. Hence $\varphi _{A,k}$ is convex.\\
Assume now that $\varphi_{A,k} $ is convex. 
Take $a^{1}, a^{2}\in A$ and $\lambda\in[0,1]$. $\Rightarrow (a^{1},0),(a^{2},0)\in \operatorname*{epi}\varphi _{A,k}$. 
$\Rightarrow (\lambda a^{1}+(1-\lambda ) a^2,0)\in \operatorname*{epi}\varphi _{A,k}$ since $\operatorname*{epi}\varphi _{A,k}$ is convex.
$\Rightarrow \lambda a^{1}+(1-\lambda ) a^2\in A$. Thus $A$ is convex.\\
$\varphi_{A,k} $ is quasiconvex if and only if $A$ is convex because of (\ref{f-r252n}).
\item[(h)] Suppose first that $A$ is a cone. Then $\operatorname*{dom}\varphi _{A,k}$ is a cone. Take $(y,t)\in \operatorname*{epi}\varphi _{A,k}$ 
 and $\lambda\in\mathbb{R}_{+}$. $\Rightarrow y\in tk+A$. 
Then $\lambda y\in \lambda tk+\lambda A\subseteq \lambda tk +A$.
$\Rightarrow \operatorname*{epi}\varphi _{A,k}$ is a cone. Hence $\varphi _{A,k}$ is positively homogeneous.\\
Assume now that $\varphi_{A,k} $ is positively homogeneous. 
Take $a\in A$ and $\lambda\in\mathbb{R}_{+}$. $\Rightarrow (a,0)\in \operatorname*{epi}\varphi _{A,k}$. 
$\Rightarrow (\lambda a,0)\in \operatorname*{epi}\varphi _{A,k}$ since $\operatorname*{epi}\varphi _{A,k}$ is a cone.
$\Rightarrow \lambda a\in A$. Thus $A$ is a cone.\\
Assume now that $A$ is a cone and $k\in -\operatorname*{core} A$. Then $Y=A+\mathbb{R}_{>}k=\operatorname*{dom}\varphi _{A,k}$.
The second part of the last statement of part (h) of the Theorem results from Theorem \ref{t251} (a) and $A-\mathbb{R}_{>}k\subseteq A+\operatorname*{core}A=\operatorname*{core}A$.
\item[(i)] Suppose first that $A+A\subseteq A$. Then $\operatorname*{dom}\varphi _{A,k}+\operatorname*{dom}\varphi _{A,k}\subseteq\operatorname*{dom}\varphi _{A,k}$.
Take $(y^{1},t_1),(y^{2},t_2)\in \operatorname*{epi}\varphi _{A,k}$. 
$\Rightarrow y^i\in t_{i}k+A$ for $i\in \{1,2\}$. Then $y^{1}+y^{2}\in (t_{1}+t_{2})k+A$.
$\Rightarrow \operatorname*{epi}\varphi _{A,k}+\operatorname*{epi}\varphi _{A,k}\subseteq\operatorname*{epi}\varphi _{A,k}$. Hence $\varphi _{A,k}$ is subadditive.\\
Assume now that $\varphi_{A,k} $ is subadditive. 
Take $a^{1}, a^{2}\in A$. $\Rightarrow (a^{1},0),(a^{2},0)\in \operatorname*{epi}\varphi _{A,k}$. 
$\Rightarrow (a^{1}+a^2,0)\in \operatorname*{epi}\varphi _{A,k}$ since $\operatorname*{epi}\varphi _{A,k}+\operatorname*{epi}\varphi _{A,k}\subseteq\operatorname*{epi}\varphi _{A,k}$.
$\Rightarrow a^{1}+a^2\in A$. Thus $A+A\subseteq A$.
\item[(j)] follows from (g) and (h) since a functional is sublinear if and only if it is convex and positively homogeneous.
\end{itemize}
\end{proof}
Moreover, we will prove in Proposition \ref{varphi_cx_proper} that $\varphi_{A,k}$ does not attain any real value if $k\in (-0^+A)\cap 0^+A$. Further statements for the important case $k\in -\operatorname*{core}0^+A$ will be given in Proposition \ref{k-core}.

\begin{remark}
Property (\ref{f-r255nn}) is called translation invariance and plays an important role in several proofs as well as for applications in risk theory. It was shown for $\varphi _{A,k}$ in \cite{GopTamZal:00}. Hamel \cite[Proposition 1]{Ham12} pointed out that each translation-invariant functional can be presented in the form $\varphi _{A,k}$ with a $k$-directionally closed set $A$ for which $k\in -0^+A$.
\end{remark}

Assumption $(H1_{A,k})$ could also be formulated in other ways.
\begin{proposition}\label{H1-alternat}
Suppose that $Y$ is a topological vector space and $A$ is a closed proper subset of $\,Y$.\\
The following conditions are equivalent to each other for $A$ and $k\in Y\setminus\{0\}$.
\begin{itemize}
\item[(a)] $k\in -0^+A$.
\item[(b)] $A=H-C$ for some proper subset $H$ of $\,Y$ and some convex cone $C\subset Y$ with $k\in C$.
\item[(c)] $A=A-C$ for some non-trivial closed convex cone $C\subset Y$ with $k\in C$.
\item[(d)] $A=A-C$ for some non-trivial cone $C\subset Y$ with $k\in C$.
\end{itemize}
\end{proposition}
\begin{proof}
\begin{itemize}
\item[]
\item[(i)] (a) implies (b) with $H=A$ and $C=-0^+A$. (b) implies (a) since $C\subseteq -0^+A$.
\item[(ii)] (a) implies (c) with $C=-0^+A$. (c) yields (d). (d) implies (a) because of $C\subseteq -0^+A$.
\end{itemize}
\end{proof}

\begin{remark}
One of the basic assumptions in production theory is the free-disposal assumption $A=A-C$, where $C$ is the ordering cone.
\end{remark}

In general, the inclusion in (\ref{r0}) is strict even for $t = 0$. The following example illustrates that $\varphi_{A,k}$ in Theorem \ref {t251} can take values other than zero on the boundary of A (finite ones as well as $-\infty$) and that $\{y \in Y\mid \varphi _{A,k} (y)<t \}=t k+\mathop{\rm int}A$ does not necessarily hold.
The example also points out that $\operatorname*{dom}\varphi _{A,k}$ may be open though $A-\mathbb{R}_{>}\cdot k\not\subseteq \operatorname*{int} \;A$.

\begin{example}\label{ex1}
Consider $Y=\mathbb{R}^{2}$, $k:=(1,0)^T$ and
$A:=\{ (y_1,y_2)^T\in\mathbb{R}^{2}\mid y_1\le -1\} \cup$ $\{ (y_1,y_2)^T\in\mathbb{R}^{2}\mid y_1\le 0, y_2\le 0\}\cup \{ (y_1,y_2)^T\in\mathbb{R}^{2}\mid y_2\le -1\}$.

Then 
\[ \varphi_{A,k}((y_1,y_2)^T)=\left\{
\begin{array}{r@{\quad\mbox{ if }\quad}l}
-\infty & y_2\le -1,\\
y_1 & -1<y_2\le 0,  \\
y_1+1 & y_2>0 .
\end{array}
\right.
\]
In particular, $\varphi_{A,k}((0,-1)^T)=-\infty$ and $\varphi_{A,k}((-1,0)^T)=-1$, though $(0,-1)^T\in\operatorname*{bd}A$ and $(-1,0)^T\in\operatorname*{bd}A$.\\
$\varphi_{A,k}$ is not continuous in $(0,-1)^T$, but $\operatorname*{dom}\varphi _{A,k}=\mathbb{R}^{2}$ is open.
\end{example}

$\varphi_{A,k}$ can be continuous though (\ref{vor-H2}) is not fulfilled.

\begin{example}\label{ex-Lip-bd}
Consider $Y=\mathbb{R}^{2}$, $k:=(1,0)^T$ and
$A:=-\mathbb{R}_{+}^{2}$.
$\varphi_{A,k}$ is continuous, but $A-\mathbb{R}_{>}\cdot k\not\subseteq \operatorname*{int} \;A$.
\end{example}

$\operatorname{core}A$ and $\operatorname{int}A$ may differ under assumption $(H1_{A,k})$.

\begin{example}\label{ex-core-not-int}
Consider the Euclidean space $Y=\mathbb{R}^{2}$, $k:=(1,0)^T$ and
$A:= \mathbb{R}^{2}\setminus \{ (y_1,y_2)^T\mid y_2\not=0,y_1>0,-y_1^2< y_2<y_1^2\}$.
Then $(H1_{A,k})$ holds, but $0\in \operatorname{core}A\setminus\operatorname{int}A$. Moreover, $0\notin\operatorname{core}(\operatorname{core}A$). 
\end{example}

$(H1_{A,k})$ and $A=\operatorname*{bd}A -\mathbb{R}_{+} k$ do not imply (\ref{r255nx-inA}) or $\operatorname{bd}A-\mathbb{R}_{>}  k\subseteq \operatorname{int}A $.
\begin{example}\label{ex-bd-line}
In $Y=\mathbb{R}^{2}$, define $k:=(1,0)^T$, $A:=\{ (y_1,y_2)^T\mid y_2=0\}$.
Then $A-[0, + \infty)\cdot k\subseteq A$ and  $A=\operatorname*{bd}A -\mathbb{R}_{+} k$ are fulfilled, but $A$ is a line parallel to $k$. Here, $\operatorname{int}A=\emptyset.$
\end{example}

The functional $\varphi_{A,k}$ has been constructed in such a way that it can be used for the separation of not necessarily convex sets. 

We need the following Lemma for the separation theorem to come.

\begin{lemma}\label{l-sep-cl}
Consider two subsets $S_1, S_2$ of a topological space $Y$.
\begin{itemize}
\item[(a)] $\operatorname*{int}S_1\cap S_2=\emptyset\iff \operatorname*{int} S_1\cap \operatorname*{cl} S_2=\emptyset$.
\item[(b)] If $S_1\subseteq \operatorname*{cl}\operatorname*{int} S_1$, then:\\
$\;\operatorname*{int} S_1\cap S_2=\emptyset\implies S_1\cap \operatorname*{int} S_2=\emptyset$.
\end{itemize}
\end{lemma}
\begin{proof}
\begin{itemize}
\item[]
\item[(a)] results immediately from the definition of interior and closure (see \cite[Satz 2.4.2]{Wei85}).
\item[(b)] $\operatorname*{int} S_1\cap S_2=\emptyset\Rightarrow\operatorname*{int} S_1\cap \operatorname*{int}S_2=\emptyset\Rightarrow\operatorname*{cl}\operatorname*{int} S_1\cap \operatorname*{int}S_2=\emptyset\Rightarrow S_1\cap \operatorname*{int}S_2=\emptyset$ since $S_1\subseteq \operatorname*{cl}\operatorname*{int} S_1$.
\end{itemize}
\end{proof}

\begin{theorem}\label{t-sep-all}
Assume $(H1_{A,k})$ and $D\subseteq Y$.
Then $\varphi_{A,k}(a)\leq 0\;\forall a\in A$ and $\varphi_{A,k}(a) < 0 \;\forall a\in \operatorname*{core} A$.
\begin{itemize}
\item[(1)] $A\cap D=\emptyset\iff  \varphi_{A,k}(d)\not\leq 0\;\forall d\in D$.
\item[(2)] $\varphi_{A,k}(d)\not< 0\;\forall d\in D\implies \operatorname*{core} A\cap D=\emptyset$.
\item[(3)] If $A-\mathbb{R}_{>}\cdot k\subseteq \operatorname*{core}A$, then :\\
$\operatorname*{core} A\cap D=\emptyset\iff \varphi_{A,k}(d)\not< 0 \;\forall d\in D$.
\end{itemize}
\end{theorem}
\begin{proof}
The first statement results from Theorem \ref{t251}. (1) and (2) follow from (\ref{f-r252n}) and (\ref{core_in_less}), respectively. (3) is implied by Theorem \ref{t251} (a).
\end{proof}

Note that $\not\leq$ and $\not<$ can only be replaced by $>$ and $\geq$, respectively, if $Y={\mathbb{R}}k+A$.

The values of $\varphi _{A,k}$ are connected with the values of $\varphi _{0^+A,k}$.
\begin{proposition}\label{p-varphi_rec}
Assume $(H1_{A,k})$. Then $(H1_{0^+A,k})$ holds.\\
For $y^0\in A+\mathbb{R}k$ and $y^1\in 0^+A+\mathbb{R}k$, we get $y^0+y^1\in A+\mathbb{R}k$ and
\begin{equation}
\varphi_{A,k}(y^0+y^1)\leq \varphi_{A,k}(y^0)+\varphi_{0^+A,k}(y^1).
\end{equation}
\end{proposition}

\begin{proof}
Since $A$ is closed, $0^+A$ is closed and $(H1_{0^+A,k})$ holds.
Consider arbitrary values $t_0, t_1\in\mathbb{R}$ for which $y^0\in A+t_0k, y^1\in 0^+A+t_1k$ is satisfied.
$\Rightarrow y^0+y^1\in A+0^+A+t_0k+t_1k\subseteq A+(t_0+t_1)k$. Thus $\varphi_{A,k}(y^0)\leq t_0$ and $\varphi_{0^+A,k}(y^1)\leq t_1$ imply $\varphi_{A,k}(y^0+y^1)\leq t_0+t_1$.
The assertion follows.
\end{proof}

Many statements which connect $\varphi_{A,k}$ with $\varphi_{0^+A,k}$ remain valid if we replace $\varphi_{0^+A,k}$ by $\varphi_{C,k}$ with $C\subseteq 0^+A$ being a closed cone with $k\in -C$ since we have:

\begin{proposition}
Assume $(H1_{A,k})$ and that $A_0$ is a proper closed subset of $A$ with $k\in -0^+ A_0$. Then $\operatorname*{dom}\varphi_{A_0,k}\subseteq\operatorname*{dom}\varphi_{A,k}$ and
$\varphi_{A,k}(y)\leq \varphi_{A_0,k}(y)\;\forall y\in\operatorname*{dom}\varphi_{A_0,k}$.
\end{proposition}

\begin{proof}
Consider some $y\in\operatorname*{dom}\varphi_{A_0,k}$. $\Rightarrow y\in\operatorname*{dom}\varphi_{A,k}$, and for each $t\in\mathbb{R}$ with $\varphi_{A_0,k}(y)\leq t$ we have
$y\in tk+A_0\subseteq tk+A$ and thus $\varphi_{A,k}(y)\leq t$.
\end{proof}

Let us now investigate the influence of the choice of $k$ on the values of $\varphi _{A, k}$.
\begin{proposition}\label{t-scale}
Assume $(H1_{A,k})$, and consider some arbitrary $\lambda\in\mathbb{R}_{>}$.
Then $(H1_{A,\lambda k})$ holds, $\operatorname*{dom}\varphi _{A,\lambda k}=\operatorname*{dom}\varphi _{A,k}$ and
\begin{displaymath}
\varphi_{A,\lambda k}(y)=\frac{1}{\lambda} \varphi_{A,k}(y) \quad \forall\, y\in Y.
\end{displaymath}
$\varphi_{A,\lambda k}$ is 
proper, finite-valued, continuous, lower semicontinuous, upper semicontinuous, convex, concave, strictly quasiconvex, subadditive, superadditive, affine, linear, sublinear, positively homogeneous, odd and homogeneous, respectively, iff $\varphi_{A,k}$ has this property. 
If $B\subset Y$, then $\varphi_{A,\lambda k}$ is 
$B$-monotone and strictly $B$-monotone, respectively,
if $\varphi_{A,k}$ has this property.
\end{proposition}

\begin{proof}
$\varphi _{A, \lambda k}(y)=\inf \{t\in {\mathbb{R}} \mid y\in t(\lambda k) + A\}
=\inf \{t\in {\mathbb{R}} \mid y\in (\lambda t) k + A\}=\inf \{\frac{1}{\lambda} u \mid u\in {\mathbb{R}}, y\in u k + A\}
=\frac{1}{\lambda}\inf \{u\in {\mathbb{R}} \mid y\in u k + A\}
=\frac{1}{\lambda}\varphi _{A, k}(y)\; \forall\, y\in Y$. 
The other assertions follow from this equation.
\end{proof}

The proposition underlines that replacing $k$ by another vector in the same direction just scales the functional. Consequently, $\varphi _{A, k}$ and $\varphi _{A, \lambda k}$, $\lambda >0$, take optimal values on some set $F\subset Y$ in the same elements of $F$. Hence it is sufficient to consider only one vector $k$ per direction in optimization problems, e.g. to restrict $k$ to unit vectors if $Y$ is a normed space.

If $\varphi _{A,k}(0)\in\mathbb{R}$, the functional can be shifted in such a way that the function value in the origin becomes zero and essential properties of the functional do not change.
\begin{proposition}\label{0-shift}
Assume $(H1_{A,k})$, and consider some arbitrary $c\in \mathbb{R}$.
Then $(H1_{A+ck,k})$ holds, $\operatorname*{dom}\varphi _{A+ck,k}=\operatorname*{dom}\varphi _{A,k}$ and
\begin{displaymath}
\varphi_{A+ck,k}(y)=\varphi_{A,k}(y)-c \quad \forall\, y\in Y.
\end{displaymath}
\end{proposition}

In vector optimization or when dealing with local ordering structures, the functional is often constructed by sets which depend on some given point $y^0$.
\begin{proposition}\label{A-shift}
Assume $(H1_{A,k})$, and consider some arbitrary $y^{0}\in Y$.
Then $(H1_{y^{0}+A,k})$ is satisfied, $\operatorname*{dom}\varphi _{y^{0}+A,k}=y^{0}+\operatorname*{dom}\varphi _{A,k}$ and
\begin{displaymath}
\varphi_{y^{0}+A,k}(y)=\varphi_{A,k}(y-y^{0}) \quad \forall\, y\in Y.
\end{displaymath}
$\varphi_{y^{0}+A,k}$ is 
proper, finite-valued, continuous, lower semicontinuous, upper semicontinuous, convex, concave, strictly quasiconvex and affine, respectively,
iff $\varphi_{A,k}$ has this property.\\
For $B\subset Y$, $\varphi_{y^{0}+A,k}$ is
$B$-monotone and strictly $B$-monotone, respectively,
iff $\varphi_{A,k}$ has this property.
\end{proposition}
\begin{proof}
$\varphi _{y^{0}+A, k}=\inf \{t\in {\mathbb{R}} \mid y\in tk + y^{0}+A\}
=\inf \{t\in {\mathbb{R}} \mid y-y^{0}\in tk + A\}
=\varphi _{A, k}(y-y^{0})$. 
The other properties follow from this result.
\end{proof}

We get the following lemma and its corollary from \cite{Zal:86a}.
\begin{lemma}\label{cone_finite_Zal}
Let $C$ be a convex cone in a linear space $Y$ and $k\in Y\setminus\{0\}$.
Then $Y=C+\mathbb{R}k$ holds if and only if
\begin{itemize}
\item[(a)] $C$ is a linear subspace of $Y$ of codimension $1$ and $k\not\in C$, or
\item[(b)] $\{k,-k\}\cap \operatorname*{core}C\not=\emptyset$. 
\end{itemize}
\end{lemma}

\begin{corollary}\label{cor_finite_Zal}
Let $C$ be a convex cone in a linear space $Y$ and $k\in C\setminus\{0\}$.
Then $Y=C+\mathbb{R}k$ holds if and only if $k\in\operatorname*{core}C$.
\end{corollary}

\begin{proposition}\label{k-core}
Assume $(H1_{A,k})$ and $k\in -\operatorname*{core}0^+A$.\\
Then $\varphi _{A, k}$ is finite-valued and $\operatorname*{lev}\nolimits_{\varphi_{A,k},<}(t)=tk+ \operatorname*{core}A \quad \forall\, t \in \mathbb{R}$.\\
Moreover, $\varphi _{0^+A, k}$ is finite-valued and
\begin{equation}
\varphi_{A,k}(y^0)-\varphi_{A,k}(y^1)\leq \varphi_{0^+A,k}(y^0-y^1) \;\;\forall y^0,y^1\in Y.\label{vor_sublin_dom}
\end{equation}
\end{proposition}
\begin{proof}
By Lemma \ref{cone_finite_Zal}, $Y=0^+A+\mathbb{R}k$, thus $Y=A+0^+A+\mathbb{R}k \subseteq A+\mathbb{R}k=\operatorname*{dom}\varphi _{A, k}$.\\
Suppose now that $\varphi _{A, k}$ is not finite-valued. Then there exists some $y\in Y$ with $y+\mathbb{R}k \subseteq A$, which implies $Y=0^+A+\mathbb{R}k=0^+A+\mathbb{R}k+y\subseteq 0^+A+A\subseteq A$, a contradiction.\\
Since $A+\operatorname*{core}0^+A\subseteq \operatorname*{core}0^+A$, the assertion related to the level sets follows from Theorem \ref{t251}(a).\\
Proposition \ref{p-varphi_rec} implies inequality (\ref{vor_sublin_dom}).
\end{proof}
\smallskip

\section{Representation of Binary Relations by Functions with Uniform Level Sets and Monotonicity}\label{sec-Monot}

Binary relations, especially partial orders, can structure a space or express preferences in decision making and  optimization. Thus the presentation of such relations by real-valued functions serves as a useful tool in proofs, e.g. in operator theory \cite{Kra64}, but also as a basis for scalarization methods in vector optimization \cite{Wei90} and for the development of risk measures in mathematical finance \cite{art99}.

If $C$ is a closed ordering cone in a topological vector space $Y$, then the corresponding order $\leq _{C}$ can be presented by $\varphi_{-C,k}$ with an arbitrary $k\in C\setminus\{0\}$ since, for all $y^1,y^2\in Y$,
\begin{displaymath}
y^1\leq _{C} y^2\iff \varphi_{-C,k}(y^1-y^2)\leq 0.
\end{displaymath}
We will show (see Corollary \ref{c251a}) that, in this case, $\varphi_{-C,k}$ is $C$-monotone, and thus we get for all $y^1,y^2\in \operatorname*{dom}\varphi_{-C,k}$:
\begin{equation}\label{mon-not-reverse}
y^1\leq _{C} y^2\implies \varphi_{-C,k}(y^1)\leq \varphi_{-C,k}(y^2).
\end{equation}

More generally, if a binary relation can be described by some proper closed subset $A$ of $Y$ with $0^+A\not= \{ 0\}$ as
$\mathcal{R}_A=\{(y^1,y^2)\in Y\times Y\mid y^2-y^1\in A\}$,  
then we have for each $k\in 0^+A\setminus \{ 0\}$ and all $y^1, y^2\in Y$:
\begin{displaymath}
y^1\mathcal{R}_A y^2\iff \varphi_{-A,k}(y^1-y^2)\leq 0.
\end{displaymath}
If this function $\varphi_{-A,k}$ is $A$-monotone, this implies for all $y^1,y^2\in \operatorname*{dom}\varphi_{-A,k}$:
\begin{displaymath}
y^1\mathcal{R}_A y^2 \implies \varphi_{-A,k}(y^1)\leq \varphi_{-A,k}(y^2).
\end{displaymath}

The reverse implication is not true since it is already not true for (\ref{mon-not-reverse}).
\begin{example}
Consider $Y=\mathbb{R}^2$, $C=\mathbb{R}^2_{+}$ and $k=(1,1)^T$. Then $\operatorname*{dom}\varphi_{-C,k}=Y$.
For $y^1=(-1,-1)^T$ and $y^2=(-2,0)^T$, we get
$\varphi_{-C,k}(y^1)=-1\leq 0=\varphi_{-C,k}(y^2)$, but $y^1\leq _{C} y^2$ does not hold since $y^2-y^1=(-1,1)^T\not\in C$.
\end{example}

But we get the following local presentation of $\mathcal{R}_A$. One has for all $y^1, y^2\in Y$:
\begin{displaymath}
y^1\mathcal{R}_A y^2\iff \varphi_{y^2-A,k}(y^1)\leq 0.
\end{displaymath}

Let us now characterize monotonicity of $\varphi_{A,k}$. We will make use of the following lemma.

\begin{lemma}\label{l-Bmon}
Let $Y$ be a linear space, $B\subseteq Y$ and $\varphi :Y \rightarrow \overline{\mathbb{R}}_{\nu }$.
\begin{itemize}
\item[(a)] If $\varphi$ is $B$-monotone or strictly $B$-monotone, then it is $A$-monotone or strictly $A$-monotone, respectively, for each subset $A$ of $B$.
\item[(b)] If $\varphi$ is strictly $B$-monotone, then it is $B$-monotone.
\item[(c)] Assume that $Y$ is a topological vector space, $B\subseteq\operatorname*{cl}\operatorname*{int}B$, $\operatorname*{dom}\varphi -\operatorname*{int}B\subseteq \operatorname*{dom}\varphi$ and
that $\varphi$ is proper and lower semicontinuous on $\operatorname*{dom}\varphi$.\\
Then $\varphi$ is $(\operatorname*{int}B)$-monotone if and only if it is $B$-monotone.
\end{itemize}
\end{lemma}

\begin{proof}
\begin{itemize}
\item[]
\item[(a)] and (b) follow from Definition \ref{d-mon}.
\item[(c)] Let $\varphi$ be $(\operatorname*{int}B)$-monotone. Consider $y^1,y^2 \in \operatorname*{dom}\varphi$ with $y^{2}-y^{1}\in B$ and assume $\varphi (y^{1})>\varphi (y^{2})$.
Since $\varphi$ is lower semicontinuous on $\operatorname*{dom}\varphi$, there exists some neighborhood U of $y^1$ such that $\varphi(y) > \varphi(y^2)\;\forall y\in U\cap \operatorname*{dom}\varphi$.
Because of $y^{1}\in y^{2}-B\subseteq y^{2}-\operatorname*{cl}\operatorname*{int}B=\operatorname*{cl}(y^{2}-\operatorname*{int}B)$ there exists some $y^3\in y^2-\operatorname*{int}B$ with $y^3\in U$. Thus $\varphi (y^{3})>\varphi (y^{2})$,
a contradiction to the $(\operatorname*{int}B)$-monotonicity of $\varphi$. Thus $\varphi$ is $B$-monotone. The assertion follows with (a). 
\end{itemize}
\end{proof}

Part (c) is due to \cite{Wei90}. The assumption $B\subseteq\operatorname*{cl}\operatorname*{int}B$ is fulfilled if $B$ is some convex set with nonempty interior.

\begin{theorem}
\label{t251M}
Assume $(H1_{A,k})$.
Then the following statements hold for sets $B\subseteq Y$:
\begin{itemize}
\item[(a)] $A-B\subseteq A \implies \varphi _{A,k}$ is B-monotone.
\item[(b)]
If $A-B\subseteq\mathbb{R}k+A$, then:\\
$\varphi _{A,k} $ is $B${-monotone} $\iff \,A-B\subseteq A$.
\item[(c)] If $\varphi _{A,k}$ is finite-valued on $F\subseteq Y$, then:\\
$A - (B\setminus \{0\}) \subseteq \operatorname{core} A \implies \varphi _{A,k} $ is strictly $B$-monotone on $F$.
\item[(d)] 
If $A-\mathbb{R}_{>} k\subseteq \operatorname*{core} \;A$ and $A-B\subseteq\mathbb{R}k+A$, then:\\
$\varphi _{A,k} $ is strictly $B${-monotone} $\implies \,A - (B\setminus \{0\}) \subseteq \operatorname{core} A$.
\item[(e)] 
If $A-\mathbb{R}_{>} k\subseteq \operatorname*{int} \;A$ and $A-B\subseteq\mathbb{R}k+A$, then:\\
$\varphi _{A,k} $ is strictly $B${-monotone} $\implies \,A - (B\setminus \{0\}) \subseteq \operatorname{int} A$.
\item[(f)] If $\varphi _{A,k}$ is proper and $A-\operatorname*{int}B\subseteq\mathbb{R}k+A$, then:\\
$\varphi _{A,k} $ is strictly $(\operatorname*{int}B)${-monotone} $\iff \varphi _{A,k} $ is $(\operatorname*{int}B)${-monotone}.\\
If additionally $B\subseteq \operatorname*{cl}\operatorname*{int}B$, then:\\
$\varphi _{A,k} $ is strictly $(\operatorname*{int}B)${-monotone} $\iff \varphi _{A,k} $ is $B${-monotone}.
\end{itemize}
\end{theorem}

\begin{proof}
\begin{itemize}
\item[]
\item[(a)] Suppose $A - B\subseteq A$. Take $y^1,y^2 \in \operatorname*{dom}\varphi_{A,k}$ with $y^{2}-y^{1}\in B$. There exists a sequence
$(t_n)_{n\in\mathbb{N}}$ which converges to $\varphi_{A,k}(y^2)$ such that $y^2\in t_nk+A \;\forall n\in\mathbb{N}$.
$\Rightarrow y^{1}\in y^{2}-B\subseteq t_nk + (A - B)\subseteq t_nk+A \;\forall n\in\mathbb{N}$. $\Rightarrow \varphi_{A,k} (y^{1})\le t_n \;\forall n\in\mathbb{N}$.
Thus $\varphi_{A,k} (y^{1})\le \varphi_{A,k}(y^2)$. Hence $\varphi_{A,k} $ is $B$-monotone.
\item[(b)] Assume now \,$A-B\subseteq{\mathbb{R}}k+A=\operatorname*{dom}\varphi _{A,k}$ and that $\varphi_{A,k} $ is $B$-monotone.
Consider $a\in A$ and $b\in B$. From
(\ref{f-r252n}) we get that $\varphi_{A,k} (a)\le 0$. Since
$a-(a-b)=b\in B$ and $a-b\in \operatorname*{dom}\varphi _{A,k}$, we obtain that $\varphi_{A,k} (a-b)\le \varphi_{A,k}
(a)\le 0$, thus $a-b\in A$ by (\ref{f-r252n}). Consequently, $A - B\subseteq A$. The assertion follows because of (a).
\item[(c)] Suppose that $\varphi _{A,k}$ is finite-valued on $F$ and $A - (B\setminus \{0\}) \subseteq \operatorname{core} A$. Take $y^1,y^2 \in F$ with $y^{2}-y^{1}\in B\setminus \{0\}$. $t:=\varphi_{A,k}(y^2)\in \mathbb{R}$. $\Rightarrow y^{2}\in tk+A$. Then $y^{1}\in y^{2}-(B\setminus \{0\}) \subseteq tk + (A - (B\setminus \{0\}))\subseteq tk+\operatorname{core} A$. 
$\Rightarrow \varphi_{A,k}(y^{1})<t=\varphi_{A,k}(y^{2})$ because of (\ref{core_in_less}). Consequently, $\varphi_{A,k}$ is strictly $B$-monotone on $F$.
\item[(e)] Assume now that $A-\mathbb{R}_{>} k\subseteq \operatorname*{int} \;A$, $A-B\subseteq {\mathbb{R}}k+A=\operatorname*{dom}\varphi _{A,k}$ and that $\varphi_{A,k} $ is strictly $B$-monotone.
Take $a\in A$ and $b\in B\setminus \{0\}$. From (\ref{f-r252n}) we get that $\varphi_{A,k} (a) \le 0$. Since
$a-(a-b)=b\in B\setminus \{0\}$ and $a-b\in \operatorname*{dom}\varphi _{A,k}$, we obtain that $\varphi_{A,k} (a-b) < \varphi_{A,k}
(a)\le 0$, thus $a-b\in \operatorname*{int} A$ by (\ref{int_in_equal}). Consequently, $A-(B\setminus \{0\})\subseteq \operatorname*{int}A$.
\item[(d)] can be proved in the same way as (e) with $\operatorname{core}$ instead of $\operatorname{int}$.
\item[(f)] $\varphi _{A,k} $ is $(\operatorname*{int}B)${-monotone} $\Rightarrow A-\operatorname*{int}B\subseteq A \Rightarrow A-\operatorname*{int}B\subseteq \operatorname*{int}A
\Rightarrow \varphi _{A,k} $ strictly $(\operatorname*{int}B)${-monotone}. The assertion follows by Lemma \ref{l-Bmon}.
\end{itemize}
\end{proof}

Note that the condition $A-B\subseteq A$ is fulfilled if $B\subseteq -0^+A$.

In parts of the previous theorems, assertions can be proved which make assumptions on the boundary of $A$ only instead of on the entire set $A$.
Let us point out that such assumptions are not necessarily weaker than conditions for the complete set $A$.

\begin{lemma}\label{bdA-equ-A}
Assume $(H1_{A,k})$, $A=\operatorname*{bd}A -\mathbb{R}_{+} k$ and $B\subset Y$. 
Then
\begin{itemize}
\item[(a)] $A+A\subseteq A\iff \,\operatorname{bd}A+\operatorname{bd}A\subseteq A$.
\item[(b)] $A-B\subseteq A\iff \,\operatorname{bd}A-B\subseteq A$.
\item[(c)] $A-(B\setminus \{0\})\subseteq \operatorname{int}A\iff \,\operatorname{bd}A-(B\setminus \{0\})\subseteq \operatorname{int} A$.
\item[(d)] $A-B\subseteq {\mathbb{R}}k+A\iff \,\operatorname{bd}A-B\subseteq {\mathbb{R}}k+A$.
\end{itemize}
\end{lemma}
\begin{proof}
\begin{itemize}
\item[]
\item[(a)] $\operatorname{bd}A+\operatorname{bd}A\subseteq A \Rightarrow A+A=(\operatorname*{bd}A -\mathbb{R}_{+} k)+(\operatorname*{bd}A -\mathbb{R}_{+} k)\subseteq
A -\mathbb{R}_{+} k\subseteq A$.
\item[(b)] $\operatorname{bd}A-B\subseteq A$ implies $A-B=(\operatorname*{bd}A -\mathbb{R}_{+} k)-B=(\operatorname*{bd}A - B)-\mathbb{R}_{+} k\subseteq A-\mathbb{R}_{+} k\subseteq A$.
\item[(c)] $\operatorname{bd}A-(B\setminus \{0\})\subseteq \operatorname{int}A$ implies $A-(B\setminus \{0\})=(\operatorname*{bd}A -\mathbb{R}_{+} k)-(B\setminus \{0\})=(\operatorname*{bd}A - (B\setminus \{0\}))-\mathbb{R}_{+} k\subseteq \operatorname{int} A-\mathbb{R}_{+} k\subseteq \operatorname{int} A$.
\item[(d)] $\operatorname{bd}A-B\subseteq {\mathbb{R}}k+A$ implies $A-B=(\operatorname*{bd}A -\mathbb{R}_{+} k)-B=(\operatorname*{bd}A - B)-\mathbb{R}_{+} k\subseteq {\mathbb{R}}k+A$.
\end{itemize}
\end{proof}

Let us summarize the results of Theorem \ref{t251M} for finite-valued functions.

\begin{corollary}\label{t251M-finite}
Assume $(H1_{A,k})$, that $\varphi _{A,k}$ is finite-valued and $B\subseteq Y$.
\begin{itemize}
\item[(a)] $\varphi _{A,k}$ is B-monotone $\iff A-B\subseteq A$.
\item[(b)] $A - (B\setminus \{0\}) \subseteq \operatorname{core} A \implies \varphi _{A,k} $ is strictly $B$-monotone.
\item[(c)] If $A-\mathbb{R}_{>} k\subseteq \operatorname*{core} \;A$, then:\\
$\varphi _{A,k} $ is strictly $B${-monotone} $\iff \,A - (B\setminus \{0\}) \subseteq \operatorname{core} A$.
\item[(d)] If $A-\mathbb{R}_{>} k\subseteq \operatorname*{int} \;A$, then:\\
$\varphi _{A,k} $ is strictly $B${-monotone} $\iff \,A - (B\setminus \{0\}) \subseteq \operatorname{int} A$.
\item[(e)] $\varphi _{A,k} $ is strictly $(\operatorname*{int}B)${-monotone} $\iff \varphi _{A,k} $ is $(\operatorname*{int}B)${-monotone}.
\item[(f)] If $B\subseteq\operatorname*{cl}\operatorname*{int}B$, then:\\
$\varphi _{A,k} $ is strictly $(\operatorname*{int}B)${-monotone} $\iff \varphi _{A,k} $ is $B${-monotone}.
\end{itemize}
\end{corollary}

Theorem \ref{t251M} contains some interesting special cases.

\begin{corollary}\label{t251M-recc}
Assume $(H1_{A,k})$ and $C\subseteq -0^+A$.
\begin{itemize}
\item[(a)] $\varphi _{A,k}$ is $C$-monotone.
\item[(b)] If  $\varphi _{A,k}$ is proper, then $\varphi _{A,k}$ is strictly $(\operatorname{int}C)$-monotone.
\end{itemize}
\end{corollary}

Thus we get by Proposition \ref{k-core} the following statement which is of basic importance in vector optimization.
\begin{corollary}
Assume $(H1_{A,k})$, that $C\subset Y$ is a non-trivial pointed convex cone in $Y$, $k\in\operatorname*{core}C$ and $C\subseteq -0^+A$.\\
Then $\varphi _{A,k}$ is finite-valued, $C$-monotone and strictly $(\operatorname{int}C)$-monotone.
\end{corollary}

Furthermore, Theorem \ref{t251M} implies:

\begin{corollary}\label{-A-mon}
Assume $(H1_{A,k})$, $\;A+A\subseteq A$ and $C\subseteq -A$.
Then $\varphi_{A,k}$ is subadditive and $C$-monotone. If $\varphi_{A,k}$ is proper, then it is strictly $(\operatorname*{int} \;C)$-monotone.
\end{corollary}

The assumptions of Corollary \ref{-A-mon} do not imply that $A$ is a cone or a shifted cone, even if $A$ is convex and $\varphi_{A,k}$ is proper.
\begin{example}
In $Y=\mathbb{R}^{2}$, consider $A:=\{(y_1,y_2)^T\in\mathbb{R}^{2}\mid y_{2}\ge \frac{1}{y_{1}}, y_{1}>0\}$ and $k:=(-1,0)^T$. 
Then $k\in-0^+A$, and $A$ is a closed convex proper subset of $Y$ for which $A+A\subseteq A$ holds and which does not contain lines parallel to $k$.
\end{example}
\smallskip

\section{Continuous Functions with Uniform Level Sets and Lipschitz Continuity}\label{sec-basicprop-cont}

We will now investigate continuous functionals $\varphi _{A,k}$ more in detail. 
Theorem \ref{t251} contains conditions which are sufficient for the continuity of $\varphi _{A,k}$, and we have discussed these conditions in section \ref{sec-basicprop}.
If $\varphi _{A,k}$ is continuous on $\operatorname*{dom}\varphi _{A,k}$ and $\operatorname*{dom}\varphi _{A,k}$ is open, then the level sets $\operatorname*{lev}_{\varphi_{A,k},=}(t)$ are given by the shifted boundary of $A$. The next theorem points out this structural dependence and summarizes additional properties of continuous functionals $\varphi _{A,k}$ with open domains.

\begin{theorem}\label{prop-funcII}
Assume
\begin{description}
\item[$(H2_{A,k})$] $Y$ is a topological vector space, $A$ is a closed proper subset of $\;Y$ and $k \in Y \setminus
\{0\}$ such that
\begin{equation}
A-\mathbb{R}_{>} k\subseteq \operatorname*{int} \;A. \label{r250a}
\end{equation}

\end{description}

Then 
\begin{equation}
A=\operatorname*{cl}(\operatorname*{int}A),\label{rem251}
\end{equation}
\begin{equation}
\operatorname*{core}A=\operatorname*{int}A,\label{H2_core_int}
\end{equation}
$\varphi _{A,k}$ is continuous on $\operatorname*{dom}\varphi _{A,k}$,
\begin{equation}
\operatorname*{dom}\varphi _{A,k} ={\mathbb{R}}k+A={\mathbb{R}}k+\operatorname*{int}A\not=  \emptyset \mbox{ is an open set},\label{f-domint}
\end{equation}  
\begin{eqnarray}
\operatorname*{lev}\nolimits _{\varphi_{A,k},\le}(t) & = & t k+A\quad \forall\, \ t \in {\mathbb{R}},\label{f-r252n-wdhl}\\
\operatorname*{lev}\nolimits _{\varphi_{A,k},<}(t) &= & t k+\mathop{\rm int}A\quad\forall\, t \in {\mathbb{R},} \label{f-r254-1nn} \\
\operatorname*{lev}\nolimits _{\varphi_{A,k},=}(t) & = & t k+\mathop{\rm bd}A\quad \forall\, t \in {\mathbb{R}}. \label{f-r254-2nn}
\end{eqnarray} 
Moreover,
\begin{itemize}
\item[(a)] $\varphi _{A,k}(y)=-\infty \iff y+\mathbb{R}k\subseteq \operatorname*{int} \;A$.
\item[(b)] $\mathop{\rm bd}A+\mathbb{R}k$ is the subset of $\;Y$ on which $\varphi _{A,k}$ is finite-valued, and
$\operatorname*{dom}\varphi _{A,k}\setminus \operatorname*{int}A\subseteq \mathop{\rm bd}A+\mathbb{R}_{+}k$. 
\item[(c)] The following conditions are equivalent:\\
\begin{equation}
\varphi _{A,k} \mbox{ is proper},
\end{equation}
\begin{eqnarray}
\operatorname*{int}A & = & \operatorname*{bd}A-\mathbb{R}_{>} k,\label{A-int-cond}\\
A & = & \operatorname*{bd}A -\mathbb{R}_{+} k,\label{A-bd-cond}\\
\operatorname*{dom}\varphi _{A,k} & = & \operatorname*{bd}A +\mathbb{R} k.\label{A-bd-cond2}
\end{eqnarray}
\item[(d)] $\varphi _{A,k} $ is finite-valued if and only if
\begin{equation}
Y=\operatorname*{bd}A+\mathbb{R}k. 
\end{equation}
\item[(e)] If $\varphi _{A,k} $ is proper, then:\\
$\varphi _{A,k} $ is strictly quasiconvex $\iff$ $A$ is a strictly convex set.
\item[(f)] If $\varphi _{A,k} $ is finite-valued, then:\\
$\varphi _{A,k} $ is linear $\iff$ $\operatorname*{bd}A$ is a linear subspace of $Y$.
\end{itemize}
\end{theorem}

\begin{proof} Suppose that $(H2_{A,k})$ holds.\\
If $y\in A$, then $y-\frac{1}{n}k\in\operatorname*{int}A$ for any $n\in \mathbb{N}_{>}$, and thus $y\in\operatorname*{cl}(\operatorname*{int}A)$. Hence $A=\operatorname*{cl}(\operatorname*{int}A)$.\\
Since $(H2_{A,k})$ implies $(H1_{A,k})$, Theorem \ref{t251} can be applied. (\ref{core_in_less}) and (\ref{int_in_equal}) yield (\ref{H2_core_int}).\\
$y\in A+\mathbb{R}k \Rightarrow \,\exists a\in A, t\in \mathbb{R}:\quad y=a+tk=a-k+(t+1)k\in \operatorname*{int}A+\mathbb{R}k$. 
Thus $A+\mathbb{R}k= \operatorname*{int}A+\mathbb{R}k\label{A-Rk-int}$.
\begin{itemize}
\item[(a)] $\varphi_{A,k}(y)=-\infty \Leftrightarrow \forall t\in\mathbb{R}:\;\varphi_{A,k}(y)< t \Leftrightarrow \forall t\in\mathbb{R}:\; y\in tk+\operatorname*{int}A \Leftrightarrow y+\mathbb{R}k\subseteq \operatorname*{int}A$.
\item[(b)] The first statement results from (\ref{f-r254-2nn}), the second one from (a) and (\ref{r250a}).
\item[(c)] Assume first that $\varphi_{A,k} $ is proper. Then for each $a\in\operatorname*{int}A$ (\ref{f-r254-1nn}) implies $\varphi_{A,k}(a)=t<0$ for some $t\in\mathbb{R}$           and by (\ref{f-r254-2nn}) $a\in tk+\operatorname*{bd}A\subseteq \operatorname*{bd}A -\mathbb{R}_{>}k$. Thus (\ref{A-int-cond}) holds because of (\ref{r250a}).\\
(\ref{A-int-cond}) implies (\ref{A-bd-cond}) by $A=\operatorname*{bd}A\cup\operatorname*{int}A$.\\
If $A=\operatorname*{bd}A -\mathbb{R}_{+} k$, then $\operatorname*{dom}\varphi _{A,k}=A +\mathbb{R} k
=\operatorname*{bd}A -\mathbb{R}_{+} k+\mathbb{R} k=\operatorname*{bd}A+\mathbb{R} k$.\\
If $\operatorname*{dom}\varphi _{A,k}=\operatorname*{bd}A +\mathbb{R} k$, then $\varphi_{A,k} $ is proper because of  
(\ref{f-r254-2nn}).
\item[(d)] results from (\ref{A-bd-cond2}).
\item[(e)] Let $\varphi _{A,k} $ be proper and $\lambda\in (0,1)$.\\
Assume first that $\varphi _{A,k} $ is strictly quasiconvex. Consider $a^1,a^2\in A$ with $a^1\not= a^2$.
$\Rightarrow$ $\varphi_{A,k}(a^1)\le 0$, $\varphi_{A,k}(a^2)\le 0$ and $\varphi_{A,k}(\lambda a^1+(1-\lambda )a^2)<\operatorname*{max}(\varphi_{A,k}(a^1),\varphi_{A,k}(a^2))\le 0$.
$\Rightarrow \lambda a^1+(1-\lambda )a^2\in\operatorname*{int}A$. Thus $A$ is a strictly convex set.\\
Assume now that $A$ is a strictly convex set. $\Rightarrow$ $\operatorname*{dom}\varphi _{A,k}$ is convex. Consider $y^1, y^2\in\operatorname*{dom}\varphi _{A,k}$ with $y^1\not= y^2$.
$t_1:= \varphi _{A,k}(y^1), t_2:= \varphi _{A,k}(y^2)$. $\Rightarrow\exists a^1,a^2\in A$: $y^1=a^1+t_1k, y^2=a^2+t_2k$.
If $a^1=a^2$, then $t_1\not= t_2$ and $\lambda y^1+(1-\lambda )y^2=a^1+(\lambda t_1+(1-\lambda )t_2)k$, which implies 
$\varphi _{A,k}(\lambda y^1+(1-\lambda )y^2)\le \lambda t_1+(1-\lambda )t_2<\operatorname*{max}(t_1,t_2)$.
If $a^1\not= a^2$, then $\lambda y^1+(1-\lambda )y^2=\lambda a^1+(1-\lambda )a^2+(\lambda t_1+(1-\lambda )t_2)k$, which, because of 
$\lambda a^1+(1-\lambda )a^2\in\operatorname*{int}A$, implies $\varphi _{A,k}(\lambda y^1+(1-\lambda )y^2)<\lambda t_1+(1-\lambda )t_2\le \operatorname*{max}(t_1,t_2)$.
Thus $\varphi _{A,k} $ is strictly quasiconvex.
\item[(f)] follows from (\ref{f-r254-2nn}) applied to $t=0$.
\end{itemize}
\end{proof}

Since $(H2_{A,k})$ implies $(H1_{A,k})$, Theorem \ref{t251} and Theorem \ref{t251M} contain further properties of the functional given in Theorem \ref{prop-funcII}.

The proposition to come will show alternative formulations of assumption $(H2_{A,k})$.
\begin{lemma}\label{l-int_inAC}
Assume that $Y$ is a topological vector space, $A\subseteq Y$ and that $C\subset Y$ is a non-trivial cone. Then \[ \operatorname*{int} \;A\subseteq A-(C\setminus \{0\}), \]
and the following conditions are equivalent:
\begin{itemize}
\item[(a)] $A-(C\setminus \{0\})\subseteq \operatorname*{int} \;A$.
\item[(b)] $A-(C\setminus \{0\})= \operatorname*{int} \;A$.
\end{itemize}
\end{lemma}
\begin{proof}
Consider some arbitrary $a\in\operatorname*{int} \;A$ and $c\in C\setminus \{0\}$. Then there exists some 
 $\lambda\in\mathbb{R}_>$ with $a^1:=a+\lambda c\in A$. $\Rightarrow \lambda c\in C\setminus \{0\}$ and $a=a^1-\lambda c$.
 Thus $\operatorname*{int} \;A\subseteq A-(C\setminus \{0\})$. Hence (a) is equivalent to (b).
\end{proof}

\begin{proposition}\label{H2-alternat}
Assume that $Y$ is a topological vector space and $A$ is a closed proper subset of $\,Y$.\\
The following conditions are equivalent to each other for $A$ and $k\in Y\setminus\{0\}$.
\begin{itemize}
\item[(a)] $A-\mathbb{R}_{>} k\subseteq \operatorname*{int}A$.
\item[(b)] $A-\mathbb{R}_{>} k = \operatorname*{int}A$.
\item[(c)] $A-(C\setminus \{0\})=\operatorname*{int}A$ for some non-trivial convex cone $C\subset Y$ with $k\in C$.
\item[(d)] $A-(C\setminus \{0\})=\operatorname*{int}A$ for some non-trivial cone $C\subset Y$ with $k\in C$.
\end{itemize}
If $\,Y$ is a Hausdorff space, then these conditions are equivalent to 
\begin{itemize}
\item[(e)] $A-(C\setminus \{0\})=\operatorname*{int}A$ for some non-trivial closed convex cone $C\subset Y$ with $k\in C$.
\end{itemize}
\end{proposition}
\begin{proof}
(a) and (b) are equivalent because of Lemma \ref{l-int_inAC}.\\
(b) implies (c) with $C=\mathbb{R}_+ k$. (c) implies (d). (a) follows from (d). \\
(e) implies (a), and if $Y$ is a Hausdorff space, then (a) implies (e) with $C=\mathbb{R}_+ k$. 
\end{proof}

\begin{remark}
The property $A-(C\setminus \{0\})=\operatorname*{int} \;A$, where $C$ is the ordering cone, is known in production theory as the strong free-disposal assumption.
\end{remark}

Even under condition $(H2_{A,k})$, $A$ is not necessarily convex and $\varphi_{A,k}$ is not necessarily proper. 

\begin{example}
 \label{ex1b} In $Y=\mathbb{R}^{2}$, define $k:=(-1,0)^T$ and $A:=\{ (y_1,y_2)^T\mid -e^{y_1}\le y_2 \le e^{y_1} \}$.
Then assumption $(H2_{A,k})$ in Theorem \ref{prop-funcII} is fulfilled, $\varphi_{A,k}$ has the domain $Y$, but is not proper since $\varphi_{A,k}(y)=-\infty$ for all $y\in Y$ with $y_2=0$.
\end{example}

The assumptions of Theorem \ref{prop-funcII} do not result in $\operatorname*{dom}\varphi _{A,k}=Y$, which will be shown in Example \ref{ex1a}.

Since $(H2_{A,k})$ implies $(H1_{A,k})$, Theorem \ref{t-sep-all} can also be applied in this case for the separation of sets. Moreover, we have:

\begin{theorem}\label{t-sep-H2}
Assume $(H2_{A,k})$ and $D\subseteq Y$.
\begin{itemize}
\item[(a)] $\operatorname*{int} A\cap D=\emptyset\iff \varphi_{A,k}(d)\not< 0 \;\forall d\in D$.
\item[(b)] $\operatorname*{int} A\cap D=\emptyset\implies A\cap \operatorname*{int} D=\emptyset\implies \varphi_{A,k}(d)\not\leq 0\;\forall d\in \operatorname*{int} D$.
\item[(c)] If $D\subseteq \operatorname*{cl}\operatorname*{int}D$, then:\\
$\operatorname*{int} A\cap D=\emptyset\iff A\cap \operatorname*{int} D=\emptyset\iff \varphi_{A,k}(d)\not\leq 0\;\forall d\in \operatorname*{int} D$.
\end{itemize}
\end{theorem}
\begin{proof}
(a) follows from (\ref{f-r254-1nn}). (b) and the reverse direction of (c) result from (1) in Theorem \ref{t-sep-all} and from Lemma \ref{l-sep-cl} since $\operatorname*{cl}\operatorname*{int} A=A$.
\end{proof}

The condition $D\subseteq \operatorname*{cl}\operatorname*{int}D$ is fulfilled if $D$ is a convex set with nonempty interior.
Again, $\not\leq$ and $\not<$ can only be replaced by $>$ and $\geq$, respectively, if $Y={\mathbb{R}}k+A$.

In many cases, properties of a function $\varphi$ become obvious by certain properties of $-\varphi$.
Keep in mind for the next proposition that, for each set $A$ in a topological space $Y$ which fulfills $(H2_{A,k})$, $\operatorname*{bd}(Y\setminus\operatorname*{int}A)=\operatorname*{bd}A$
and $Y\setminus\operatorname*{int}(Y\setminus\operatorname*{int}A)=A$.

\begin{proposition}\label{-phi_Ak}
Assume $(H2_{A,k})$. Then $(H2_{Y\setminus\operatorname*{int}A,-k})$ holds and
\begin{itemize}
\item[(a)] $\operatorname*{bd}A +\mathbb{R} k$ is the subset of $Y$ on which $\varphi _{A,k}$ is finite-valued as well as the subset of $Y$ on which 
$\varphi _{Y\setminus \operatorname*{int}A,-k}$ is finite-valued,
\item[(b)] $\operatorname*{dom}\varphi _{A,k}\cap\operatorname*{dom}\varphi _{Y\setminus \operatorname*{int}A,-k}=\operatorname*{bd}A +\mathbb{R} k$,
\item[(c)] $\varphi _{A,k}(y)=-\varphi _{Y\setminus\operatorname*{int}A,-k}(y)\;\;\forall y\in \operatorname*{bd}A +\mathbb{R} k.$
\end{itemize}
\end{proposition}
\begin{proof}
Obviously, $(H2_{A,k})$ implies $(H2_{Y\setminus\operatorname*{int}A,-k})$. 
\begin{itemize}
\item[(a)] Because of Theorem \ref{prop-funcII}, $\varphi _{A,k}$ is finite-valued on $\operatorname*{bd}A +\mathbb{R} k$ and
$\varphi _{Y\setminus\operatorname*{int}A,-k}$ is finite-valued on $\operatorname*{bd}(Y\setminus\operatorname*{int}A) +\mathbb{R} k=\operatorname*{bd}A +\mathbb{R} k$.
\item[(b)] $y\in\operatorname*{dom}\varphi _{A,k}\setminus (\operatorname*{bd}A +\mathbb{R} k) \Leftrightarrow \varphi _{A,k}(y)=-\infty \Leftrightarrow y+\mathbb{R} k\subseteq\operatorname*{int}A$.
The last inclusion implies $y\not\in (Y\setminus \operatorname*{int}A)-\mathbb{R} k=\operatorname*{dom}\varphi _{Y\setminus \operatorname*{int}A,-k}$. 
\item[(c)] We get by (\ref{f-r254-2nn}) for all $y\in \operatorname*{bd}A +\mathbb{R} k$:
$\;t=\varphi _{A,k}(y) \Leftrightarrow y\in tk+\operatorname*{bd}A \Leftrightarrow y\in (-t)\cdot (-k)+\operatorname*{bd}(Y\setminus\operatorname*{int}A)
\Leftrightarrow -t=\varphi _{Y\setminus\operatorname*{int}A,-k}(y).$
\end{itemize}
\end{proof}

\begin{remark}
Proposition \ref{-phi_Ak} gives us a tool to transfer results related to the minimization of functions of type $\varphi_{A,k}$ to the maximization of functions of this type. 
\end{remark}

\begin{example} \label{ex1a}
Consider $Y=\mathbb{R}^{2}$, $k:=(1,0)^T$ and $A:=\{ (y_1,y_2)^T\mid y_2\ge e^{y_1}\}$.
Then $(H2_{A,k})$ is fulfilled and $\operatorname*{dom}\varphi _{A,k} ={\mathbb{R}}k+A=\{ (y_1,y_2)^T\mid y_2 > 0\}$.
$\varphi_{A,k}$ is a proper functional.\\
$\operatorname*{dom}\varphi _{Y\setminus\operatorname*{int}A,-k}=Y$, but the set on which $\varphi _{Y\setminus\operatorname*{int}A,-k}$ is finite-valued is $\operatorname*{dom}\varphi _{A,k}$, where $\varphi _{Y\setminus\operatorname*{int}A,-k}=-\varphi _{A,k}$ holds.
\end{example}

Proposition \ref{-phi_Ak}, Theorem \ref{t251} and Theorem \ref{prop-funcII} result in the following statements.
\begin{proposition}\label{phi_Ak_cv}
Assume $(H2_{A,k})$ and $Y=\operatorname*{bd}A+\mathbb{R}k$. \\
Then 
$\varphi _{A,k}:Y\rightarrow \mathbb{R}$ given by $\varphi_{A,k} (y)= \inf \{t\in
{\mathbb{R}} \mid y\in tk + A\}$ is
\begin{itemize}
\item[(a)] concave $\iff Y\setminus\operatorname*{int}A$ is convex,
\item[(b)] strictly quasiconcave $\iff Y\setminus\operatorname*{int}A$ is a strictly convex set,
\item[(c)] superadditive $\iff (Y\setminus\operatorname*{int}A)+(Y\setminus\operatorname*{int}A)\subseteq Y\setminus\operatorname*{int}A$.
\end{itemize}
\end{proposition}
In economics, utility functions are often assumed to be strictly quasiconcave.

We now turn to sets $A$ for which the recession cone has a nonempty interior and $-k$ is an element of this interior. 
Proposition \ref{k-core} implies:

\begin{proposition}\label{prop-finite}
Assume
\begin{description}
\item[$(H3_{A,k})$] $Y$ is a topological vector space, $A$ is a closed proper subset of $\;Y$ and $k\in -\operatorname*{int}0^+A$.
\end{description}
Then $(H2_{A,k})$ holds and $\varphi_{A,k}$ is continuous and finite-valued.\\
Moreover, $(H3_{0^+A,k})$ holds and (\ref{vor_sublin_dom}).
\end{proposition}

Condition $(H3_{A,k})$ could also be formulated alternatively.
\begin{proposition}\label{H3-alternat}
Assume that $Y$ is a topological vector space and $A$ is a closed proper subset of $\,Y$.\\
The following conditions are equivalent to each other for $A$ and $k\in Y$.
\begin{itemize}
\item[(a)] $k\in -\operatorname*{int}0^+A$.
\item[(b)] $A=H-C$ for some proper subset $H$ of $\,Y$ and some non-trivial convex cone $C\subset Y$ with $k\in\operatorname*{int}C$.
\item[(c)] $A=A-C$ for some non-trivial closed convex cone $C\subset Y$ with $k\in\operatorname*{int}C$.
\item[(d)] $A=A-C$ for some non-trivial cone $C\subset Y$ with $k\in\operatorname*{int}C$.
\item[(e)] $A-\operatorname*{int}C\subseteq A$ for some closed convex cone $C\subset Y$ with $k\in\operatorname*{int}C$.
\item[(f)] $A-\operatorname*{int}C\subseteq A$ for some cone $C\subset Y$ with $k\in\operatorname*{int}C$.
\end{itemize}
\end{proposition}
\begin{proof}
\begin{itemize}
\item[]
\item[(i)] (a) implies (b) with $H=A$ and $C=-0^+A$. (b) implies (a) since $C\subseteq -0^+A$.
\item[(ii)] (a) implies (c) with $C=-0^+A$. (c) yields (d). (d) implies (a) because of $C\subseteq -0^+A$.
\item[(iii)] (c) yields (e). (e) implies (f). (f) implies (a).
\end{itemize}
\end{proof}

If $A$ is a convex cone, then $(H3_{A,k})$ is equivalent to $(H2_{A,k})$. But in general, $(H2_{A,k})$ does not imply $k\in -\operatorname*{int} 0^+A$. $\operatorname*{int} 0^+A$ may be empty, even if $A$ is convex, has a non-empty interior and $\varphi_{A,k}$ is finite-valued.
\begin{example}
In $Y=\mathbb{R}^{2}$, define $k:=(0,-1)^T$ and $A:=\{ (y_1,y_2)^T\mid y_2\ge y_1^{2}\}$.
Then assumption $(H2_{A,k})$ in Theorem \ref{prop-funcII} is fulfilled. $\varphi_{A,k}$ is finite-valued and $A$ is convex, but $0^+A=\{-k\}$.
\end{example}

$(H3_{A,k})$ can hold though $A$ is not convex and $A+A\not\subseteq A$.
\begin{example}
In $Y=\mathbb{R}^{2}$, define $k:=(-1,-1)^T$ and $A:=(\mathbb{R}_+^2+(-1,0)^T)\cup(\mathbb{R}_+^2 +(0,-1)^T)$.
Then $(H3_{A,k})$ is fulfilled, but $A$ is not convex and $A+A\not\subseteq A$ since $(-1,0)^T+(0,-1)^T\not\in A$.
\end{example}

We are now going to investigate the Lipschitz continuity of $\varphi_{A,k}$.

\begin{theorem}\label{th-Lip-varphi}
Assume that $Y$ is a Banach space and $A$ a proper closed subset of $Y$ for which there exists some $k\in-0^+A\setminus\{0\}$.
Then:\\
$\varphi_{A,k}$ is finite-valued and Lipschitz $\iff \varphi_{0^+A,k}$ is finite-valued and Lipschitz $\iff k\in -\operatorname*{int}0^+A$.\\
If $k\in-\operatorname*{int}0^+A$, then each Lipschitz constant of $\varphi_{0^+A,k}$ is also a Lipschitz constant of $\varphi_{A,k}$.
\end{theorem}
\begin{proof}
The assumption implies $(H1_{A,k})$ and $(H1_{0^+A,k})$.
\begin{itemize}
\item[(a)] Assume first $k\in-\operatorname*{int}0^+A$.
By Proposition \ref{prop-finite}, $\varphi_{A,k}$ and $\varphi_{0^+A,k}$ are continuous and finite-valued.
Since $0^+A$ is a convex cone, $\varphi_{0^+A,k}$ is sublinear and thus Lipschitz with some Lipschitz constant $L$.
Assume now $y^0, y^1\in Y$ and choose the notation such that $\varphi_{A,k}(y^0)\geq \varphi_{A,k} (y^1)$.
$\Rightarrow 0\leq \varphi_{A,k}(y^0)-\varphi_{A,k}(y^1)\leq \varphi_{0^+A,k}(y^0-y^1)$ by (\ref{vor_sublin_dom}).
$\Rightarrow |\varphi_{A,k}(y^0)-\varphi_{A,k}(y^1)|\leq |\varphi_{0^+A,k}(y^0-y^1)|=|\varphi_{0^+A,k}(y^0-y^1)-\varphi_{0^+A,k}(0)|\leq L||y^0-y^1||$.
Hence $\varphi_{A,k}$ is Lipschitz with Lipschitz constant $L$.
\item[(b)] Assume now that $\varphi_{A,k}$ is finite-valued and Lipschitz with some Lipschitz constant $L$.
$U:=\{y\in Y\mid ||y||\leq \frac{1}{L}\}$ is a neighborhood of $0$. Consider arbitrary elements $u\in U$, $a\in A$ and $\lambda\in\mathbb{R}_+$.
$|\varphi_{A,k}(a+\lambda u)-\varphi_{A,k}(a)|\leq L||\lambda u||=L\lambda ||u||\leq \lambda$.
$\Rightarrow \varphi_{A,k}(a+\lambda u)\leq \varphi_{A,k}(a)+\lambda$. $\Rightarrow \varphi_{A,k}(a+\lambda (u-k))=\varphi_{A,k}(a+\lambda u)-\lambda\leq \varphi_{A,k}(a)\leq 0$.
$\Rightarrow a+\lambda (u-k)\in A$. $\Rightarrow u-k\in 0^+A\;\forall u\in U$. $\Rightarrow k\in -\operatorname*{int}0^+A$.
\item[(c)] Assume that $\varphi_{0^+A,k}$ is finite-valued and Lipschitz. Apply (b) to $0^+A$ instead of $A$. This yields 
$k\in -\operatorname*{int}0^+(0^+A)=-\operatorname*{int}0^+A$.
\end{itemize}
\end{proof}

The function $\varphi_{A,k}$ given in Example \ref{ex-Lip-bd} is Lipschitz, but not finite-valued \cite{TamZal10}. There, of course, $k\not\in -\operatorname*{int}0^+A$. Tammer and Z{\u{a}}linescu proved, in more general spaces than those in Theorem \ref{th-Lip-varphi}, that, under the assumptions given there, $\varphi_{A,k}$ is finite-valued and Lipschitz if and only if $k\in -\operatorname*{int}0^+A$ \cite[Corollary 3.4]{TamZal10}.

In Proposition \ref{Fact4} we will show that $\varphi_{A,k}$ is locally Lipschitz on $\operatorname*{int}\operatorname*{dom}\varphi_{A,k}$ if $A$ is a proper closed convex subset of a Banach space $\;Y$ and $k\in (-0^+A)\setminus 0^+A$. Tammer and Z\u{a}linescu \cite{TamZal10} gave an example of a function $\varphi_{A,k}$ which is not locally Lipschitz on $\operatorname*{int}\operatorname*{dom}\varphi_{A,k}$.
Moreover, \cite{TamZal10} contains conditions under which $\varphi_{A,k}$ is continuous at some point or Lipschitz on some neighborhood of a point. These conditions depend on the classical definition $\operatorname*{inf}\emptyset=+\infty$ .
\smallskip

\section{Convex Functions with Uniform Level Sets}\label{sec-cx}

The previous sections contain many properties of $\varphi_{A,k}$ which are also essential if $\varphi_{A,k}$ is convex, which is just the case for $A$ being a convex set.
Let us now give further results for this special case. We first turn to convex cones $A$ before considering more general cases.

In many applications, the set $A$ in the definition of the functional $\varphi_{A,k}$ is a non-trivial convex cone since it is then closely related to the cone order (cp. Section \ref{sec-Monot}). As pointed out in \cite{foe02}, for functionals $\varphi _{A,k}$ used in the formulation of risk measures, $A$ is the so-called acceptance set and just the ordering cone in a function space $L^p$. This cone has an empty interior.

\begin{corollary}\label{c251a}
Assume that $Y$ is a topological vector space, $A\subset Y$ is a non-trivial closed convex cone and $k\in -A\setminus\{0\}$.
Then $(H1_{A,k})$ holds and $\varphi_{A,k}$ is a sublinear $(-A)$-monotone functional which is lower semicontinuous on its convex effective domain.
\begin{itemize}
\item[(a)] If $k\in (-A)\cap A$, then $\operatorname*{dom}\varphi_{A,k}=A$ and $\varphi_{A,k}$ does not attain any real value.
\item[(b)] If $k\in (-A) \setminus A$, then $\varphi_{A,k}$ is proper and strictly $(-\operatorname*{core} A)$-monotone.
\item[(c)] $\varphi_{A,k}$ is finite-valued iff $k\in -\operatorname*{core}A$.\\
$k\in -\operatorname*{core}A$ implies $\operatorname*{lev}\nolimits_{\varphi_{A,k},<}(t)=tk+ \operatorname*{core}A \quad \forall\, t \in \mathbb{R}.$
\item[(d)] If $k\in -\operatorname*{int} A$, then $(H2_{A,k})$ holds and $\varphi_{A,k}$ is continuous and finite-valued.\\
$(H2_{A,k})$ holds if and only if $k\in -\operatorname*{int} A$.
\end{itemize}
\end{corollary}
\begin{proof} 
$(H1_{A,k})$ holds since $A+A\subseteq A$. By Theorem \ref{t251} and Theorem \ref{t251M}, $\varphi_{A,k}$ is a sublinear $(-A)$-monotone functional which is lower semicontinuous on its effective domain.
\begin{itemize}
\item[(a)] $k\in (-A)\cap A$ implies $a+\mathbb{R}k\subseteq A\;\forall a\in A$ and thus $\operatorname*{dom}\varphi_{A,k}=\mathbb{R}k+A=A$
and $\varphi_{A,k}(y)=-\infty\;\forall y\in \operatorname*{dom}\varphi_{A,k}$ by Theorem \ref{t251}.
\item[(b)] Consider some arbitrary $a\in A$.
Since $A$ is closed and $k\not\in A$, there exists some neighborhood $U$ of $k$ such that $U\subseteq Y\setminus A$.
$\Rightarrow\exists n\in\mathbb{N}_{>}:\;\;\; y:=k+\frac{1}{n}(a+k)\in U$ and thus $y\not\in A$.
$\Rightarrow ny=a+(n+1)k\not\in A$ since $A$ is a cone. Hence $\varphi_{A,k}$ is proper.
The monotonicity results from Theorem \ref{t251M}.
\item[(c)] The first sentence of (d) results from Corollary \ref{cor_finite_Zal} and (b), the second one from Theorem \ref{t251}(h).
\item[(d)] It is easy to verify that $(H2_{A,k})$ holds if and only if $k\in -\operatorname*{int} A$.
If $k\in -\operatorname*{int} A$, then $(H3_{A,k})$ holds and the assertion follows from Proposition \ref{prop-finite}.
\end{itemize}
\end{proof}

Further sufficient conditions for the continuity of $\varphi_{A,k}$ under the assumptions of Corollary \ref{c251a} will be given in Proposition \ref{Fact4}.

The functional $\varphi_{A,k}$ is often used locally in such a way that we attach some set to a point $y^0$ and use this set as $A$. Let us summarize properties of $\varphi_{y^0-C,k}$
for convex cones $C$.
\begin{corollary}\label{c251a-C}
Assume that $Y$ is a topological vector space, $C\subset Y$ is a non-trivial closed convex cone, $k\in C\setminus\{0\}$ and $y^0\in Y$.
Then $(H1_{y^0-C,k})$ holds and $\varphi_{y^0-C,k}$ is a convex $C$-monotone functional which is lower semicontinuous on its convex effective domain.
\begin{itemize}
\item[(a)] If $k\in C\cap (-C)$, then $\operatorname*{dom}\varphi_{y^0-C,k}=y^0-C$ and $\varphi_{y^0-C,k}$ does not attain any real value.
\item[(b)] If $k\in C \setminus (-C)$, then $\varphi_{y^0-C,k}$ is proper and strictly $(\operatorname*{core} C)$-monotone.
\item[(c)] If $k\in \operatorname*{core} C$, then $\varphi_{y^0-C,k}$ is finite-valued.
\item[(d)] If $k\in \operatorname*{int} C$, then 
$(H2_{y^0-C,k})$ holds and $\varphi_{y^0-C,k}$ is continuous, finite-valued and strictly $(\operatorname*{int} C)$-monotone.
\item[(e)] $\varphi_{y^0-C,k}$ is subadditive $\iff y^0\in -C$.
\item[(f)] $\varphi_{y^0-C,k}$ is sublinear $\iff y^0\in C\cap(-C)$.
\end{itemize}
\end{corollary}
\begin{proof} 
Apply Corollary \ref{c251a} to $A=-C$. (a)-(d) follow from Proposition \ref{A-shift}.
\begin{itemize}
\item[(e)] $\varphi_{y^0-C,k}$ is subadditive $\Leftrightarrow (y^0-C)+(y^0-C)\subseteq y^0-C \Leftrightarrow y^0+(y^0-C)\subseteq y^0-C$ since $C+C=C$. Hence  $\varphi_{y^0-C,k}$ is subadditive $\Leftrightarrow y^0-C\subseteq -C \Leftrightarrow y^0\in -C$.
\item[(f)] Since $y^0-C$ is convex, we have:
$\;y^0-C$ is a convex cone $\Leftrightarrow (0\in y^0-C$ and $(y^0-C)+(y^0-C)\subseteq y^0-C)$ $\Leftrightarrow (y^0\in C$ and $y^0\in -C)$ because of (e).
\end{itemize}
\end{proof}

If $\varphi_{A,k}$ is convex, then it is proper or does not attain any real value.

\begin{proposition}\label{varphi_cx_proper}
Assume $(H1_{A,k})$.
\begin{itemize}
\item[(a)] If $k\in (-0^+A)\cap 0^+A$, then $\varphi_{A,k}$ does not attain any real value.
\item[(b)] If $k\in (-0^+A)\setminus 0^+A$ and $A$ is convex, then $\varphi_{A,k}$ is proper.
\end{itemize}
\end{proposition}
\begin{proof}
\begin{itemize}
\item[]
\item[(a)] $(H1_{A,k})$ implies $(H1_{0^+A,k})$. Applying Corollary \ref{c251a} to $0^+A$ instead of $A$ results in $\varphi_{0^+A,k}(0)=-\infty$. The assertion follows from the inequality in Proposition \ref{p-varphi_rec}
 with $y^1=0$.
\item[(b)] Assume that $k\in -0^+A$, $A$ is convex, but $\varphi_{A,k}$ is not proper.
$\Rightarrow\exists a^0\in A:\;a^0+\mathbb{R}k\subseteq A$. Consider arbitrary elements $a\in A$, $t\in\mathbb{R}_{>}$.
$\;a^1:=a+tk$. Let $U$ be an arbitrary neighborhood of $a^1$. $\Rightarrow \exists n\in\mathbb{N}_{>}:\; y:=a^1+\frac{1}{n}(a^0-a)\in U$.
$ y= a+tk+\frac{1}{n}(a^0-a)=(1-\frac{1}{n})a+\frac{1}{n}(a^0+ntk)\in A$ since $A$ is convex. Thus $a^1\in\operatorname*{cl}A=A$ and $a+\mathbb{R}_+ k\subseteq A\;\forall a\in A$. $\Rightarrow k\in 0^+A$. 
\end{itemize} 
\end{proof}

Note that the recession cone of each nonempty closed convex unbounded set in a finite dimensional separated topological vector space does not contain only the zero vector \cite{Zal:02}. But Z\u{a}linescu \cite[Example 1.1.1]{Zal:02} gave an example for an unbounded closed convex set $A$ in $\ell ^p$, $p\in [1,\infty]$, with $0^+A=\{0\}$.

\begin{proposition}\label{Fact4}
Assume that $Y$ is a topological vector space, $A$ is a proper closed convex subset of $\;Y$ and $k\in (-0^+A)\setminus 0^+A$.\\
Then $(H1_{A,k})$ is fulfilled and the following statements are valid.
\begin{itemize}
\item[(a)] $\varphi _{A,k}$ is convex, proper and lower semicontinuous on $\operatorname*{dom}\varphi _{A,k}$.
\item[(b)] $\operatorname*{dom}\varphi _{A,k}=A+\mathbb{R}\cdot k$ is convex.
\item[(c)] $\varphi _{A,k}$ is finite-valued iff $\;Y=\operatorname*{dom}\varphi _{A,k}$. 
\item[(d)] $\operatorname*{dom}\varphi _{A,k}=\operatorname*{core}\operatorname*{dom}\varphi _{A,k}$ iff $A-\mathbb{R}_{>}\cdot k\subseteq \operatorname*{core} \;A$.
\item[(e)] If $\operatorname*{int}A\not=\emptyset$ or $Y=\mathbb{R}^n$, then $(H2_{A,k})$ holds iff
$\operatorname*{dom}\varphi _{A,k}$ is open.
\item[(f)] If $\operatorname*{int}A\not=\emptyset$, then $\varphi _{A,k}$ is continuous on $\operatorname*{int}\operatorname*{dom}\varphi _{A,k}$.
\item[(g)] If $\;Y$ is a Banach space, then $\varphi_{A,k}$ is locally Lipschitz on $\operatorname*{int}\operatorname*{dom}\varphi_{A,k}$.
\end{itemize}
\end{proposition}

\begin{proof}
\begin{itemize}
\item[]
\item[(a)] $\varphi _{A,k}$ is convex and lower semicontinuous on $\operatorname*{dom}\varphi _{A,k}$ by Theorem \ref{t251}, proper by Proposition \ref{varphi_cx_proper}.
\item[(b)] $\operatorname*{dom}\varphi _{A,k}=A+\mathbb{R}\cdot k$ is convex since $A$ is convex.
\item[(c)] results from (a).
\item[(d)]
Assume $\operatorname*{dom}\varphi _{A,k}=\operatorname*{core}\operatorname*{dom}\varphi _{A,k}$, and consider some arbitrary $y^1\in A-\mathbb{R}_{>}\cdot k$ and $y\in Y$.
$\Rightarrow y^1\in A$ since $k\in -0^+A$, and $\exists\; a^1\in A, t_1\in\mathbb{R}_{>}:\; y^1=a^1-t_1k$.
$y^1\in A+\mathbb{R}\cdot k=\operatorname*{core}( A+\mathbb{R}\cdot k)$. $\Rightarrow \exists\; t_2\in\mathbb{R}_{>}:\; y^2:=y^1+t_2(y-y^1)\in A+\mathbb{R}\cdot k$.
If $y^2\in A$, then the convexity of $A$ implies $y^1+t(y-y^1)\in A\;\forall\; t\in \mathbb{R}_{+}$ with $t\leq t_2$.
Suppose now $y^2\not\in A$. $\Rightarrow \exists\; a^2\in A, t_3\in\mathbb{R}_{>}:\; y^2=a^2+t_3k$.
$\Rightarrow y^1+\frac{t_1t_2}{t_1+t_3}(y-y^1)=y^1+\frac{t_1}{t_1+t_3}(y^2-y^1)=
a^1-t_1k+\frac{t_1}{t_1+t_3}(a^2+t_3k-(a^1-t_1k))=a^1+\frac{t_1}{t_1+t_3}(a^2-a^1)\in A$ since $A$ is convex.
Because of $y^1\in A$, we get $y^1+t(y-y^1)\in A\;\forall\; t\in \mathbb{R}_{+}$ with $t\leq \frac{t_1t_2}{t_1+t_3}$. Thus $y^1\in \operatorname*{core}A$.\\
The reverse implication of (d) follows from Theorem \ref{t251} (a). 
\item[(e)] The assumptions of (e) imply $\operatorname*{int}A = \operatorname*{core}A$ (see \cite{Hol75}, \cite{BorVan10})
and $\operatorname*{int}\operatorname*{dom}\varphi _{A,k} = \operatorname*{core}\operatorname*{dom}\varphi _{A,k}$.
Hence (d) implies (e). 
\item[(f)] If $\operatorname*{int}A\not=\emptyset$, then $\operatorname*{int}\operatorname*{dom}\varphi_{A,k}\not=\emptyset$ and $\varphi _{A,k}$ is bounded above by zero on $\operatorname*{int}A$ because of (\ref{f-r252n}). Thus $\varphi _{A,k}$ is continuous on $\operatorname*{int}\operatorname*{dom}\varphi _{A,k}$ by \cite[Theorem 5.43]{AliBor06}. 
\item[(g)] follows from Proposition 4.1.4 and Proposition 4.1.5 in \cite{BorVan10} since $\varphi _{A,k}$ is convex and lower semicontinuous on $\operatorname*{dom}\varphi _{A,k}$.
\end{itemize}
\end{proof}

Tammer and Z{\u{a}}linescu \cite[Example 3.6]{TamZal10} gave an example of a non-trivial closed convex cone $A$ in $\mathbb{R}^3$ and $k\in (-A)\setminus A$ for which $\varphi_{A,k}$ is not Lipschitz on $\operatorname*{int}\operatorname*{dom}\varphi_{A,k}$. 
Example \ref{ex1} illustrates that the statements (c)-(g) of Proposition \ref{Fact4} are not valid for an arbitrary non-convex set $A$ which fulfills $(H1_{A,k})$.

We get from Proposition \ref{Fact4} the following statement which was also proved in \cite{GopRiaTamZal:03}.

\begin{corollary}\label{corAII}
Assume
that $Y$ is a topological vector space, $A$ is a proper closed convex subset of $\;Y$, $\operatorname*{int}\; A\neq \emptyset$, $k\in -0^+A$ and $Y=\mathbb{R}k+A$.\\
Then $(H2_{A,k})$ holds and $\varphi_{A,k}$ is convex, continuous and finite-valued.
\end{corollary}

The assumptions of Corollary \ref{corAII} yield $k\in (-0^+A)\setminus 0^+A$ since otherwise $Y=A$.

Proposition \ref{prop-finite} implies a necessary condition for subdifferentials of $\varphi_{A,k}$.
\begin{proposition}\label{p-subdiff}
Assume that $Y$ is a separated locally convex space, $A$ a proper closed subset of $\;Y$ and $k\in -\operatorname*{int}0^+A$. Then $\varphi_{A,k}$
is a finite-valued continuous convex function, $\varphi_{0^+A,k}$
is a finite-valued continuous sublinear function, and we get for each $\overline{y}\in Y$:
$$y^{\ast }\in \partial\varphi_{A,k}(\overline{y})\implies y^{\ast }(y-\overline{y})\leq \varphi_{0^+A,k}(y-\overline{y})\;\forall y\in Y.$$
\end{proposition}

Note that, by Proposition \ref{A-shift}, $\varphi_{0^+A,k}(y-\overline{y})=\varphi_{\overline{y}+0^+A,k}(y)$.

A formula for the conjugate of $\varphi _{A,k}$ is given in \cite{RocUryZab02} and \cite{RusSha06}, further statements about the subdifferential of $\varphi _{A,k}$ have been proved by Durea and Tammer \cite{DT}.
\smallskip

\section{Functionals with uniform level sets, the Minkowski functional and norms}\label{s-Mink-lev}

Functions with uniform level sets which are generated by cones often coincide with a Minkowski functional on a subset of the space. Let $p_A$ denote the Minkowski functional generated by a set $A$ in a linear space.

For the proofs of the following propositions, we need two lemmata from \cite{AliBor06}.

\begin{lemma}\label{l-Mink-semi}
Assume that $p:Y\rightarrow\mathbb{R}$ is a nonnegative function on a linear space $Y$.\\
$p$ is a seminorm if and only if it is the Minkowski functional of a balanced convex absorbing set $A\subseteq Y$.
\end{lemma}

\begin{lemma}\label{l-Mink-semi2}
Suppose that $Y$ is a topological vector space.
\begin{itemize}
\item[(a)] If $A$ is a closed absorbing set which is star-shaped about zero, then $A=\{y\in Y\mid p_A(y)\leq 1\}$.
\item[(b)] A real-valued  nonnegative function on $Y$
is sublinear and lower semicontinuous if and only if it is the Minkowski functional of an absorbing closed convex set.
\item[(c)] A real-valued nonnegative function on $Y$
is sublinear and continuous if and only if it is the Minkowski functional of a convex neighborhood of zero.
\end{itemize}
\end{lemma}

\begin{proposition}\label{p-Mink-uni}
Assume $Y$ to be a topological vector space, $C\subset Y$ a non-trivial closed convex cone and $k\in -\operatorname*{core}C$.
For the Minkowski functional $p_{C+k}$, we get
\[ p_{C+k}(y)=\left\{
\begin{array}{c@{\quad\mbox{ if }\quad}l}
\varphi_{C,k}(y) & y\in Y\setminus C,  \\
0 & y\in C,
\end{array}
\right.
\]
i.e., $$p_{C+k}(y)=\operatorname*{max}\{\varphi_{C,k}(y),0\}\quad\forall y\in Y.$$
$p_{C+k}$ is sublinear and lower semicontinuous.
\end{proposition}

\begin{proof}
By Proposition \ref{k-core}, $\varphi_{C,k}$ is finite-valued.\\
For each $y\in Y$, $p_{C+k}(y)=\operatorname*{inf}\{\lambda >0\mid y\in\lambda (C+k)\}=\operatorname*{inf}\{\lambda >0\mid y\in C+\lambda k\}$. Hence $p_{C+k}(y)=\varphi_{C,k}(y)$ if $\varphi_{C,k}(y)>0$. This is just the case for $y\in Y\setminus C$.\\
$C=C-\lambda k+\lambda k\subseteq C+\lambda k\;\forall\lambda >0$.
Hence $p_{C+k}(y)=0\;\forall y\in C$ and $p_{C+k}(y)=\operatorname*{max}\{\varphi_{C,k}(y),0\}\;\forall y\in Y$.
Since $C+k$ is convex, closed and absorbing, $p_{C+k}$ is sublinear and lower semicontinuous by Lemma \ref{l-Mink-semi2}.
\end{proof}

We are now going to investigate the relationship between functions with uniform level sets and norms which are defined by the Minkowski functional of an order interval.

From now on, let us assume that $Y$ is a linear space ordered by a non-trivial pointed convex cone $C \subset Y$ with nonempty core. The partial order $``\leq_C"$ is given by
\begin{eqnarray*}
\forall y^1, y^2\in Y:\quad y^1\leq_{C} y^2 \quad \mbox{if and only if}\quad y^2 - y^1\in C.
\end{eqnarray*}

Order intervals $[-k,k]_C$ have the following properties.
\begin{lemma}\label{l-ordint}
For each $k\in C$, the order interval $[-k,k]_C=\{y\in Y\mid -k \leq_C y \leq_C k\}=(C-k)\cap(k-C)$ is a convex balanced set.
It is an absorbing set if and only if $k\in\operatorname*{core}C$.
\end{lemma}

We can generate norms by order cones. The related seminorm and norm given in the following proposition was constructed in \cite[Lemma 1.45]{Jah86a}. The result contained in part (e) and (f) for locally convex spaces with $k\in\operatorname*{int}C$ can be found in \cite[Example 2.2.13]{KhaTamZal15}. 

\begin{proposition}\label{p-norm-ordint}
Suppose that $Y$ is a linear space, $C \subset Y$ a non-trivial convex pointed cone and $k\in\operatorname*{core}C$. Let $p$ be the Minkowski functional of the set $[-k,k]_C$. We will denote $p$ by $\|\cdot \| _{C,k}$ whenever it is a norm.
\begin{itemize}
\item[(a)] $p$ is a seminorm. 
\item[(b)] If $C$ is algebraically closed, $\|\cdot \| _{C,k}$ is a norm and $[-k,k]_C=\{y\in Y\mid \| y\| _{C,k}\leq 1\}$.
\item[(c)] If $C$ has a weakly compact base, then $\|\cdot \| _{C,k}$ is a norm and $(Y,\|\cdot \| _{C,k})$ is a reflexive normed space. Thus $(Y,\|\cdot \| _{C,k})$ is also a Banach space and an Asplund space in this case.
\item[(d)] If $Y$ is a finite-dimensional topological vector space, then $p$ is continuous.
\item[(e)] If $Y$ is a topological vector space and $C$ is closed, then $\|\cdot \| _{C,k}$ is a lower semicontinuous norm. $[-k,k]_C$ is the closed unit ball w.r.t. the norm $\|\cdot \| _{C,k}$, and $C$ has a non\-empty interior in the normed space $(Y, \| \cdot \| _{C,k})$.
\item[(f)] If the assumptions of (e) and $k\in\operatorname*{int}C$ hold, then $\|\cdot \| _{C,k}$ is continuous.
\end{itemize}
\end{proposition}
\begin{proof}
\begin{itemize}
\item[]
\item[(a)] $A:=[-k,k]_C$ is a convex balanced absorbing set. The Minkowski functional 
$p=p_{A}:Y\rightarrow\mathbb{R}$ of $A$ is a seminorm by Lemma \ref{l-Mink-semi}. 
\item[(b)] was proved by Jahn \cite[Lemma 1.45]{Jah86a}.
\item[(c)] The first sentence was proved by Jahn \cite[Lemma 1.45]{Jah86a}, the second one immediately results from it.
\item[(d)] follows from the sublinearity of $p$.
\item[(e)] The norm property and the unit ball result from (b). Since
 $A$ is the closed unit ball w.r.t. the norm $p_A$ and $k+A\subseteq C$,  
$k$ is an interior point of $C$ w.r.t. the topology generated by the norm $p_A$. $p_{A}$ is lower semicontinuous by Lemma \ref{l-Mink-semi2}
\item[(f)] Since $A$ is a convex neighborhood of zero, $p_{A}$ is continuous by Lemma \ref{l-Mink-semi2}.
\end{itemize}
\end{proof}

\begin{remark}
Let $Y$ be a Riesz space with the ordering cone $C$. Then we can show that $k\in Y$ is an order unit of $Y$ if and only if
$k\in\operatorname*{core}C$. The usual order unit norm which is often denoted by $|| \cdot ||_{\infty}$ is just $|| \cdot ||_{C,k}$. For details concerning Riesz spaces, see e.g. \cite{AliBor06}.
\end{remark}

Let us note that Proposition \ref{p-norm-ordint} yields the following statement.

\begin{corollary}
If $Y$ is a topological vector space and $C \subset Y$ a non-trivial closed convex pointed cone with nonempty core,
then this core is the interior of $C$ in some norm topology of $Y$.
\end{corollary}

\begin{proposition}\label{p-ordint-varphi}
Suppose that $Y$ is a topological vector space, $C \subset Y$ a non-trivial closed convex pointed cone with $k\in\operatorname*{core}C$, $a\in Y$. Denote by $||\cdot ||_{C,k}$ the norm which is given as the Minkowski functional of the order interval $[-k,k]_C$. Then
\[ \label{eq-norm-varphi}
||y-a||_{C,k}=\varphi_{a-C,k}(y)\quad\forall y\in a+C.
\]
\end{proposition}

\begin{proof}
Consider some $y\in a+C$.
\begin{eqnarray*}
||y-a||_{C,k} & = & \operatorname*{inf}\{\lambda >0\mid y-a\in \lambda ((C-k)\cap (k-C))\}\\
& = & \operatorname*{inf}\{\lambda >0\mid y-a\in  (C-\lambda k)\cap (\lambda k-C)\}\\
& = & \operatorname*{inf}\{\lambda >0\mid y-a\in  \lambda k-C\} \mbox{ since } y-a\in C\subseteq C-\lambda k\;\forall\lambda\in\mathbb{R}_+.\\
& = & \varphi_{a-C,k}(y) \mbox{ if } y\notin a-C.
\end{eqnarray*}
$(a+C)\cap (a-C)=\{a\}$ since $C$ is pointed. Hence $||y-a||_{C,k}=\varphi_{a-C,k}(y)\quad\forall y\in a+C$ with $y\not=a$. $||a-a||_{C,k}=0=\varphi_{a-C,k}(a)$. 
\end{proof}

In many applications, solutions are determined by problems $\operatorname*{min}_{y\in F}\| y-a\| _{C,k}$ with $F\subseteq a+C$. Replacing $\| y-a\| _{C,k}$ by $\varphi_{a-C,k}(y)$, this approach can often be applied without the assumption $F\subseteq a+C$. This is illustrated for the scalarization of vector optimization problems with the weighted Chebyshev norm and with extensions of this norm in \cite{Wei90} and \cite{Wei94}.
\smallskip

\section{Bibliographical Notes}\label{sec-phiAk_biogr}

Because of the strong connection to partial orders pointed out at the beginning of Section \ref{sec-Monot}, functions of type $\varphi _{A,k}$ have been used in proofs in different fields of mathematics for the construction of sublinear functionals. In these cases, $A$ is a closed pointed convex cone, usually the ordering cone of the space considered, and $k\in -A$. Among the earliest references listed in \cite{Ham12} are \cite{Bon54} and \cite{Kra64}, where the functional was applied in operator theory. $\varphi _{A,k}$ has also been studied in economic theory and finance, e.g. as so-called shortage function by Luenberger \cite{luen92} and for risk measures by Artzner et al. \cite{art99}. 

Tammer (formerly Gerstewitz and Gerth) was the first one who introduced a functional of type $\varphi _{A,k}$ by formula (\ref{funcak0}) in vector optimization theory, for the definition of a set of properly efficient points and the investigation of dual problems \cite{ger85}. Here, $A$ is supposed to be a closed convex set with $Y=\mathbb{R}k+A$,
$K$ is the ordering cone of the space $Y$, $k\in K\setminus\{0\}$. Then strict $K$-monotonicity is stated under the assumption $A-(K\setminus\{0\}))\subset \operatorname*{int}A$ in a Hausdorff topological vector space and $K$-monotonicity under the assumption $A-K\subseteq A$ in a barrelled locally convex space. Moreover, in both cases convexity, continuity and condition (\ref{f-r252n}) were investigated.
 
Z{\u{a}}linescu \cite{Zal:87a} considered $\varphi _{A,k}$ as a function which maps into $\overline{\mathbb{R}}$ and not necessarily into $\mathbb{R}$, under the assumption that $A$ is a proper closed convex subset of a Hausdorff topological vector space with $0\in\operatorname*{bd}A$ and $k\in (-0^+A)\cap (-\operatorname*{int}A)$. 

In \cite{GerWei90}, C. Tammer and the author investigated functionals 
$\varphi _{A,k}$ under the following alternative assumptions in an arbitrary topological vector space $Y$:
\begin{itemize}
\item[(a)] $A$ is the closure of some proper open convex subset of $Y$, $k\in -0^+A$ and $Y=A+\mathbb{R}k$;
\item[(b)] $A$ is the closure of some proper subset $C$ of $Y$ for which there exists a cone $K\subset Y$ with $\operatorname*{int}K\not= \emptyset$ and $A-\operatorname*{int}K\subseteq \operatorname*{int}C$, $k\in\operatorname*{int}K$. 
\end{itemize}
Under assumption (a) as well as under assumption (b), they proved that $\varphi _{A,k}$ is continuous and finite-valued, that the statements (\ref{f-r252n-wdhl}), (\ref{f-r254-1nn}) and (\ref{f-r254-2nn}) hold,
and gave sufficient conditions for monotonicity, strict monotonicity and subadditivity of the functional which referred to the boundary of $A$. 
They proved strict $(\operatorname*{int}K)$-monotonicity  of $\varphi _{A,k}$ under assumption (b), convexity of $\varphi _{A,k}$ under assumption (a) and showed that $\varphi _{A,k}$ is also convex under assumption (b) if $A$ is convex. The case where $\varphi _{A,k}$ becomes a sublinear function, i.e., where $C=K$ is a cone in (b), was also investigated. Tammer and Weidner formulated separation theorems based on these results. 

In \cite{Wei90a} and \cite{Wei90}, the author compared properties of functionals on topological vector spaces which are defined by one of the conditions (\ref{f-r252n}), (\ref{f-r254-1nn}), (\ref{f-r254-2nn}), by condition (\ref{f-r254-2nn}) with $A$ instead of $\operatorname*{bd}A$, in the form $\varphi _{A,k}$, or in the form $\varphi _{\operatorname*{cl}A,k}$ under the most general assumptions. She investigated under which conditions these definitions result in well-defined functionals as well as the properties of these functionals. She proved that a function which is defined by (\ref{f-r252n}) for a closed set $A$ and $k\in Y\setminus\{0\}$ is proper if and only if (\ref{r255nx-inA}) and $(H1_{A,k})$ hold
\cite[Satz 3.1.10]{Wei90}. She showed that these conditions are fulfilled for each closed pointed convex cone $A$ with $k\in -A\setminus\{0\}$.
Weid\-ner proved the equivalence of Theorem \ref{t251M} (b) for any function which fulfills (\ref{f-r252n}).
Moreover, she studied lower semicontinuity, monotonicity, convexity and quasiconvexity of such a functional. The author \cite[Satz 3.2.1]{Wei90} proved that the functionals
given by one of the conditions (\ref{f-r252n}), (\ref{f-r254-1nn}) or (\ref{f-r254-2nn}) and $\varphi _{A,k}$ coincide for a closed set $A$ on $A+\mathbb{R}k$ and are well-defined if and only if $\operatorname*{int}A=\operatorname*{bd}A-\mathbb{R}_{>}k$ holds. She stu\-died sets with this property and showed that (\ref{rem251}) holds for these sets.
Under the supposition that
\begin{displaymath}
\operatorname*{int}A=\operatorname*{bd}A-\mathbb{R}_{>}k \mbox{ and } Y=\operatorname*{bd}A+\mathbb{R}k, 
\end{displaymath} 
she proved that $\varphi _{A,k}$ is finite-valued and continuous and gave necessary and sufficient conditions for the following properties: monotonicity, strict monotonicity, convexity, quasiconvexity, strict quasiconvexity, subadditivity, superadditivity, concavity and its modifications, positive-homogeneity, sublinearity, oddness, homogeneity and linearity. It was shown that this supposition is fulfilled if assumption (a) from \cite{GerWei90} mentioned above or $(H3_{A,k})$ is satisfied, and that $(H3_{A,k})$ is equivalent to the conditions stated in Proposition \ref{H3-alternat}, which also includes the above assumption (b) for closed sets $C$.

In \cite{Wei90}, the results for $\varphi _{A,k}$ are also applied to vector optimization. Conditions for efficiency, weak efficiency and proper efficiency are derived, especially for non-convex vector optimization problems. The author \cite{Wei90} was the first one who pointed out that the scalarization by Pascoletti and Serafini \cite{PasSer84} is equivalent to the problem  
\begin{displaymath}
\operatorname*{min} \{\varphi _{a-C,k}(y)\mid y\in F\cap\operatorname*{dom}\varphi _{a-C,k}\},
\end{displaymath} 
where $F$ is the feasible point set in the finite-dimensional linear image space of a multicriteria optimization problem and $C$ is a closed convex cone in this space.
Moreover, she extended this approach to arbitrary sets $C$ in not necessarily finite-dimensional linear spaces and studied its properties. In contrary to Pascoletti and Serafini, she applied the scalar optimization problem not only for finding solutions to vector optimization problems the optima of which are defined w.r.t. a domination cone $D=C$, but admitted sets $D$ which are not cones and sets $C\not= D$.
She showed that many known scalarizations in multicriteria optimization turn out to be special cases of this general optimization problem and thus can be formulated using functionals $\varphi _{A,k}$. In this way, properties of $\varphi _{A,k}$ implied statements for different scala\-rizations which were afterwards used by 
Eichfelder \cite{eich08} for developing adaptive scala\-rization methods. 

In \cite[Theorem 2.3.1]{GopRiaTamZal:03}, $\varphi _{A,k}$ is considered as an extended-real-valued functional which is not necessarily finite-valued. Under assumption $(H1_{A,k})$, lower semicontinuity, (\ref{f-r252n}), (\ref{f-r255nn}) and necessary and sufficient conditions for $\varphi _{A,k}$ to be convex, subadditive, positive-homogeneous, proper and finite-valued, respectively, were proved. Under assumption $(H2_{A,k})$, continuity of $\varphi _{A,k}$, (\ref{f-r254-1nn}), (\ref{f-r254-2nn}) and (\ref{rem251}) were shown. \cite{GopRiaTamZal:03} also contains Lemma \ref{bdA-equ-A} (a)-(c) assuming the stronger condition $(H2_{A,k})$ and some additional assumptions.  In \cite{TamZal10}, (\ref{int_in_less}) and (\ref{r0}) are also mentioned and examples illustrate that the inclusion in (\ref{r0}) is, in general, strict \cite[Example 2.1]{TamZal10}, and that $(H2_{A,k})$ does not imply properness of $\varphi _{A,k}$ \cite[Example 3.2]{TamZal10}. Moreover, part (a) of Theorem \ref{t251M} and part (c) of Theorem \ref{t251M} for $\operatorname*{int}$ instead of $\operatorname*{core}$ were shown \cite[Theorem 3.1]{TamZal10}. They showed Proposition \ref{Fact4}(g) under the additional assumption $\operatorname*{int}A\not=\emptyset$.

Those results of this report for which no reference is given in this section or in the previous sections are original results by the author.

Let us finally note that there exist hundreds of articles which apply the functions discussed in this paper, mainly based on \cite{GerWei90}. Especially the comprehensive contribution of Christiane Tammer to this field should be mentioned. In this section we only discussed references which are immediately related to the results of this paper and point out the basic connection to scalarization in vector optimization. 

\smallskip

\def\cfac#1{\ifmmode\setbox7\hbox{$\accent"5E#1$}\else
  \setbox7\hbox{\accent"5E#1}\penalty 10000\relax\fi\raise 1\ht7
  \hbox{\lower1.15ex\hbox to 1\wd7{\hss\accent"13\hss}}\penalty 10000
  \hskip-1\wd7\penalty 10000\box7}
  \def\cfac#1{\ifmmode\setbox7\hbox{$\accent"5E#1$}\else
  \setbox7\hbox{\accent"5E#1}\penalty 10000\relax\fi\raise 1\ht7
  \hbox{\lower1.15ex\hbox to 1\wd7{\hss\accent"13\hss}}\penalty 10000
  \hskip-1\wd7\penalty 10000\box7}
  \def\cfac#1{\ifmmode\setbox7\hbox{$\accent"5E#1$}\else
  \setbox7\hbox{\accent"5E#1}\penalty 10000\relax\fi\raise 1\ht7
  \hbox{\lower1.15ex\hbox to 1\wd7{\hss\accent"13\hss}}\penalty 10000
  \hskip-1\wd7\penalty 10000\box7}
  \def\cfac#1{\ifmmode\setbox7\hbox{$\accent"5E#1$}\else
  \setbox7\hbox{\accent"5E#1}\penalty 10000\relax\fi\raise 1\ht7
  \hbox{\lower1.15ex\hbox to 1\wd7{\hss\accent"13\hss}}\penalty 10000
  \hskip-1\wd7\penalty 10000\box7}
  \def\cfac#1{\ifmmode\setbox7\hbox{$\accent"5E#1$}\else
  \setbox7\hbox{\accent"5E#1}\penalty 10000\relax\fi\raise 1\ht7
  \hbox{\lower1.15ex\hbox to 1\wd7{\hss\accent"13\hss}}\penalty 10000
  \hskip-1\wd7\penalty 10000\box7}
  \def\cfac#1{\ifmmode\setbox7\hbox{$\accent"5E#1$}\else
  \setbox7\hbox{\accent"5E#1}\penalty 10000\relax\fi\raise 1\ht7
  \hbox{\lower1.15ex\hbox to 1\wd7{\hss\accent"13\hss}}\penalty 10000
  \hskip-1\wd7\penalty 10000\box7}
  \def\cfac#1{\ifmmode\setbox7\hbox{$\accent"5E#1$}\else
  \setbox7\hbox{\accent"5E#1}\penalty 10000\relax\fi\raise 1\ht7
  \hbox{\lower1.15ex\hbox to 1\wd7{\hss\accent"13\hss}}\penalty 10000
  \hskip-1\wd7\penalty 10000\box7}
  \def\cfac#1{\ifmmode\setbox7\hbox{$\accent"5E#1$}\else
  \setbox7\hbox{\accent"5E#1}\penalty 10000\relax\fi\raise 1\ht7
  \hbox{\lower1.15ex\hbox to 1\wd7{\hss\accent"13\hss}}\penalty 10000
  \hskip-1\wd7\penalty 10000\box7}
  \def\cfac#1{\ifmmode\setbox7\hbox{$\accent"5E#1$}\else
  \setbox7\hbox{\accent"5E#1}\penalty 10000\relax\fi\raise 1\ht7
  \hbox{\lower1.15ex\hbox to 1\wd7{\hss\accent"13\hss}}\penalty 10000
  \hskip-1\wd7\penalty 10000\box7} \def\Dbar{\leavevmode\lower.6ex\hbox to
  0pt{\hskip-.23ex \accent"16\hss}D}
  \def\cfac#1{\ifmmode\setbox7\hbox{$\accent"5E#1$}\else
  \setbox7\hbox{\accent"5E#1}\penalty 10000\relax\fi\raise 1\ht7
  \hbox{\lower1.15ex\hbox to 1\wd7{\hss\accent"13\hss}}\penalty 10000
  \hskip-1\wd7\penalty 10000\box7} \def\cprime{$'$}
  \def\Dbar{\leavevmode\lower.6ex\hbox to 0pt{\hskip-.23ex \accent"16\hss}D}
  \def\cfac#1{\ifmmode\setbox7\hbox{$\accent"5E#1$}\else
  \setbox7\hbox{\accent"5E#1}\penalty 10000\relax\fi\raise 1\ht7
  \hbox{\lower1.15ex\hbox to 1\wd7{\hss\accent"13\hss}}\penalty 10000
  \hskip-1\wd7\penalty 10000\box7} \def\cprime{$'$}
  \def\Dbar{\leavevmode\lower.6ex\hbox to 0pt{\hskip-.23ex \accent"16\hss}D}
  \def\cfac#1{\ifmmode\setbox7\hbox{$\accent"5E#1$}\else
  \setbox7\hbox{\accent"5E#1}\penalty 10000\relax\fi\raise 1\ht7
  \hbox{\lower1.15ex\hbox to 1\wd7{\hss\accent"13\hss}}\penalty 10000
  \hskip-1\wd7\penalty 10000\box7}
  \def\udot#1{\ifmmode\oalign{$#1$\crcr\hidewidth.\hidewidth
  }\else\oalign{#1\crcr\hidewidth.\hidewidth}\fi}
  \def\cfac#1{\ifmmode\setbox7\hbox{$\accent"5E#1$}\else
  \setbox7\hbox{\accent"5E#1}\penalty 10000\relax\fi\raise 1\ht7
  \hbox{\lower1.15ex\hbox to 1\wd7{\hss\accent"13\hss}}\penalty 10000
  \hskip-1\wd7\penalty 10000\box7} \def\Dbar{\leavevmode\lower.6ex\hbox to
  0pt{\hskip-.23ex \accent"16\hss}D}
  \def\cfac#1{\ifmmode\setbox7\hbox{$\accent"5E#1$}\else
  \setbox7\hbox{\accent"5E#1}\penalty 10000\relax\fi\raise 1\ht7
  \hbox{\lower1.15ex\hbox to 1\wd7{\hss\accent"13\hss}}\penalty 10000
  \hskip-1\wd7\penalty 10000\box7} \def\Dbar{\leavevmode\lower.6ex\hbox to
  0pt{\hskip-.23ex \accent"16\hss}D}
  \def\cfac#1{\ifmmode\setbox7\hbox{$\accent"5E#1$}\else
  \setbox7\hbox{\accent"5E#1}\penalty 10000\relax\fi\raise 1\ht7
  \hbox{\lower1.15ex\hbox to 1\wd7{\hss\accent"13\hss}}\penalty 10000
  \hskip-1\wd7\penalty 10000\box7} \def\Dbar{\leavevmode\lower.6ex\hbox to
  0pt{\hskip-.23ex \accent"16\hss}D}
  \def\cfac#1{\ifmmode\setbox7\hbox{$\accent"5E#1$}\else
  \setbox7\hbox{\accent"5E#1}\penalty 10000\relax\fi\raise 1\ht7
  \hbox{\lower1.15ex\hbox to 1\wd7{\hss\accent"13\hss}}\penalty 10000
  \hskip-1\wd7\penalty 10000\box7}

\end{document}